\documentclass[10pt]{article} 
\usepackage{amsmath,amsfonts,amsthm}
\usepackage[a4paper]{geometry}

\newcommand{\C}{\mathbb{C}}
\newcommand{\M}{\mathcal{M}}
\newcommand{\R}{{\mathbb{R}}}
\newcommand{\He}{\mathbb{H}}
\newcommand{\cC}{\mathcal{C}^\infty_c}
\newcommand{\cCP}{\mathcal{P}^\infty}
\newcommand{\lag}{\langle}
\newcommand{\rag}{\rangle}
\newcommand{\Ga}{\Gamma}
\newcommand{\Gad}{\Gamma_{\!\!2}}

\newcommand{\Bx}{{\mathbf x}}
\newcommand{\By}{{\mathbf y}}
\newcommand{\Bp}{{\mathbf p}}
\newcommand{\bu}{\bullet}

\newtheorem{thm}{Theorem}[section]
\newtheorem{rmk}[thm]{Remark}
\newtheorem{lem}[thm]{Lemma}
\newtheorem{cor}[thm]{Corollary}

\newtheorem{prop}[thm]{Proposition}

\author{D. Bakry, F. Baudoin, M. Bonnefont,  D. Chafa\"\i\\
{\small Institut de Math\'ematiques de Toulouse} \\
{\small Universit\'e de Toulouse} \\
{\small CNRS 5219}} 

\title{On gradient bounds for the heat kernel\\ on the Heisenberg group}

\begin{document}
\maketitle

\begin{abstract}  
 It is known that the couple formed by the two dimensional Brownian motion
 and its L\'evy area leads to the heat kernel on the Heisenberg group, which
 is one of the simplest sub-Riemannian space. The associated diffusion
 operator is hypoelliptic but not elliptic, which makes difficult the
 derivation of functional inequalities for the heat kernel. However, Driver
 and Melcher and more recently H.-Q. Li have obtained useful gradient bounds
 for the heat kernel on the Heisenberg group. We provide in this paper simple
 proofs of these bounds, and explore their consequences in terms of
 functional inequalities, including Cheeger and Bobkov type isoperimetric
 inequalities for the heat kernel.
\end{abstract}

{\small\noindent\textbf{Keywords:} Heat kernel ; Heisenberg group ; functional inequalities ; hypoelliptic diffusions\\\textbf{AMS-MSC:} 22E30 ; 60J60}

{\footnotesize\tableofcontents}

\section{Introduction}

Gradient bounds had proved to be a very efficient tool for the control of the
rate of convergence to equilibrium, quantitative estimates on the
regularization properties of heat kernels, functional inequalities such as
Poincar\'e, logarithmic Sobolev, Gaussian isoperimetric inequalities for heat
kernel measures. The reader may take a look for instance at
\cite{Tanigushi,varopoulos,varopoulos-et-al,ledoux-zurich,livre} and
references therein. When dealing with the simplest examples, such as linear
parabolic evolution equations (or heat kernels), those gradient bounds often
rely on the control of the intrinsic Ricci curvature associated to the
generator of the heat kernel. Those methods basically require some form of
ellipticity of the generator.

\subsection*{The elliptic case}

Let $\M$ be a complete connected Riemannian manifold of dimension $n$ and let
$L$ be the associated Laplace-Beltrami operator, written in a local system of
coordinates as
$$
L(f)(x)= \sum_{i,j=1}^n a_{i,j}(x) \partial^{2}_{x_i x_j}f(x).
$$ 
The coefficients $x\mapsto a_{i,j}(x)$ are smooth and the symmetric matrix
$(a_{i,j}(x))_{1\leq i,j\leq n}$ is positive definite for every $x$. The
``length of the gradient'' $\left|\nabla f\right|$ of a smooth $f:\M\to\R$ is
given by
$$
\Ga(f,f)= \left|\nabla f\right|^{2} = \frac{1}{2}(L(f^{2})-2 fLf)= 
\sum_{i,j=1}^n a_{i,j}(x) \partial_{x_i}f \partial_{x_j} f.
$$
Let $(P_{t})_{t\geq0}=(e^{tL})_{t\geq0}$ be the heat semigroup generated by
$L$. For every smooth $f:\M\to\R$, the function $(t,x)\mapsto P_t(f)(x)$ is
the solution of the heat equation associated to $L$ 
$$
\partial_t P_t(f)(x)=LP_t(f)(x) \quad\text{and}\quad P_0(f)(x)=f(x).
$$
For every real number $\rho\in\R$, the following three propositions are
equivalent (see \cite{Tanigushi,ledoux-zurich,sturm}).
\begin{enumerate}
\item $\forall f\in\cC(\M),~\mathrm{Ricci}(\nabla f,\nabla f)\geq
 \rho\left|\nabla f\right|^2$
\item $\forall f\in\cC(\M),~ \forall t\geq0, ~ \left|\nabla P_{t}
   f\right|^{2}\leq e^{-2\rho t}\,P_{t}(\left|\nabla f\right|^{2})$
\item $\forall f\in\cC(\M), ~ \forall t\geq0,~ \left|\nabla P_{t}
   f\right|\leq e^{-\rho t}\,P_{t}(\left|\nabla f\right|)$
\end{enumerate}
This is the case for some $\rho\in\R$ when $\M$ is compact. This is also the
case with $\rho=0$ when $\M$ is $\R^n$ equipped with the usual metric since
$\mathrm{Ricci}\equiv0$. In this last example, $L$ is the usual Laplace
operator $\Delta$ and the explicit formula for the heat kernel gives $\nabla
P_t f=P_t \nabla f$ for the usual gradient $\nabla$ and thus $\left|\nabla P_t
 f\right|\leq P_t\left|\nabla f\right|$. Back to the general case, and
following \cite{Tanigushi}, the gradient bounds 2. or 3. above are equivalent
to their infinitesimal version at time $t=0$, which reads
$$
\Gad(f,f)\geq \rho \Ga(f,f)
$$
where
$$
\Gad(f,f) = \frac{1}{2}(L\Ga(f,f)- 2\Ga(f,Lf))
= \left|\nabla\nabla f\right|^{2}+ \mathrm{Ric}(\nabla f, \nabla f).
$$
The bound $\Gad\geq \rho \Ga$ had proved to be a very efficient criterion for
the derivation of gradient bounds for more general Markov processes, including
for instance processes generated by an operator $L$ with a first order linear
part (i.e. with a potential).

In the equivalence above, one may add several other inequalities, including
local Poincar\'e inequalities, local logarithmic Sobolev inequalities, and
local Bobkov isoperimetric inequalities, and their respective reverse forms,
with a specific constant involving $e^{-\rho t}$, see
\cite{Tanigushi,ledoux-zurich}. Here the term local means that they concern
the probability measure $P_t(\cdot)(x)$ for any fixed $t$ and $x$. One may
also replace in these inequalities $e^{-\rho t}$ by any function $c(t)$
continuous and differentiable at $t=0$ with $c(0)=1$ and $c'(0)=-\rho$. In the
present paper, we will focus on the Heisenberg group, a non elliptic situation
where these equivalences do not hold, but where some gradient bounds are still
available and provide local inequalities of various types.

\subsection*{The Heisenberg group}

In recent years, some focus had been set on some degenerate situations, where
the methods used for the elliptic case do not apply. One of the simplest
example of such a situation is the Heisenberg group (see section
\ref{se:facts} for the group structure). Namely, we consider on $\He=\R^{3}$
the vector fields
$$
X= \partial_{x}-\frac{y}{2}\partial_{z} \text{\quad and\quad} %
Y = \partial_{y} + \frac{x}{2}\partial_{z}
$$
and the operator 
\begin{align}\label{eq:LH}
L&=
X^{2}+Y^{2} 
=\partial_x^2+\partial_y^2+\frac{1}{4}(x^2+y^2)\partial_z^2+x\partial^2_{y,z}-y\partial^2_{x,z}.
\end{align}
This operator is self-adjoint for the Lebesgue measure on $\R^{3}$. The matrix
of second order derivatives associated to $L$ is degenerate and thus $L$ is
not elliptic. If $[U,V]=UV-VU$ stands for the commutator of $U$ and $V$, then
$$
Z:=[X,Y]= \partial_{z} %
\text{\quad and\quad} %
[X,Z]= [Y,Z]=0.
$$
In particular, $L$ is hypoelliptic in the H\"ormander sense (the Lie algebra
described by $\{X,Y,Z\}$ is the Lie algebra of the Heisenberg group, see
section \ref{se:facts}). As a consequence, the heat semigroup
$(P_{t})_{t\geq0} =(e^{tL})_{t\geq0}$ obtained by solving the heat equation
associated to $L$ admits a smooth density with respect to the Lebesgue measure
on $\R^3$. It is remarkable that the Markov process associated to this
semigroup is the couple formed by a Brownian motion on $\R^2$ and its L\'evy
area, and for every fixed $t>0$ and $\Bx\in\He$, the probability distribution
$P_t(\cdot)(\Bx)$ is a sort of Gaussian on $\He$. We refer to \cite{baudoin}
and \cite{neuenschwander} for such probabilistic aspects. For this operator
$L$ we have also
\begin{equation}\label{eq:GaH}
\Ga(f,f) = X(f)^{2}+ Y(f)^{2}
\end{equation}
and 
\begin{multline*}
 \Gad(f,f)= X^{2}(f)^{2}+ Y^{2}(f)^{2}+\frac{1}{2}(XY+YX)(f)^{2}
 +\frac{1}{2}(Zf)^{2}+2(XZ (f)Y(f)-YZ(f)X(f)).
\end{multline*}
The presence of $YZ(f)$ and $XZ(f)$ in the $\Gad$ expression forbids the
existence of a constant $\rho\in\R$ such that $\Gad \geq \rho \Ga$ as
functional quadratic forms. Therefore the methods used in the elliptic case to
prove gradient bounds could not work. In other words, the Ricci tensor is
everywhere $-\infty$. In fact, a closer inspection of the Ricci tensor of the
elliptic operator $X^{2}+Y^{2}+ \epsilon Z^{2}$ when $\epsilon$ goes to $0$
shows that this operator has everywhere a Ricci tensor which is
$$
\begin{pmatrix}-\frac{1}{2\epsilon} &0&0\\ 0&-\frac{1}{2\epsilon}&0\\ 0&0&\frac{1}{2\epsilon ^ 2}\end{pmatrix}.
$$
In the limit, one may consider that the lower bound of the Ricci tensor for
$L$ is everywhere $-\infty$. Despite this singularity, B. Driver and T.
Melcher proved in \cite{DriverMelcher} the existence of a finite positive
constant $C_2$ such that
\begin{equation} \label{DM} \forall f\in\cCP(\He), ~\forall t\geq0, ~
 \left|\nabla P_{t} f\right|^{2}\leq C_2 P_{t}( \left|\nabla f\right|^{2}).
\end{equation} 
where $\cCP(\He)$ is the class of smooth function from $\He$ to $\R$ with all
partial derivatives of polynomial growth. Here $C_2$ is the best constant,
i.e. the smallest possible. As in the elliptic case, the gradient bound
\eqref{DM} implies a Poincar\'e inequality for $P_t$, since
\begin{equation}\label{eq:PI-from-DM}
 P_t(f^2)-(P_t f)^2 %
 = 2\int_0^t\!P_s(\left|\nabla P_{t-s} f\right|^2)\,ds 
 \leq 2tC_2 P_t(\left|\nabla f\right|^2). 
\end{equation}
The gradient bound \eqref{DM} gives also a reverse Poincar\'e inequality for
$P_t$, since
\begin{equation}\label{eq:RPI-from-DM}
 P_t(f^2)-(P_t f)^2 %
 = 2\int_0^t\!P_s(\left|\nabla P_{t-s} f\right|^2)\,ds \geq \frac{2t}{C_2}
 \left|(\nabla P_t f)\right|^2. 
\end{equation}
>From the point of view of regularization, \eqref{eq:RPI-from-DM} is the most
important, while \eqref{eq:PI-from-DM} is more concerned with estimates on the
heat kernel and concentration properties. More recently, H.-Q. Li showed in
\cite{HQLi} that there exists a finite positive constant $C_1$ such that
\begin{equation} \label{HQLi} \forall f\in\cCP(\He), ~\forall t\geq0, ~
 \left|\nabla P_{t} f\right|\leq C_1 P_{t}( \left|\nabla f\right|).
\end{equation}
It is shown in \cite{DriverMelcher} that $C_1\geq\sqrt{2}$ and $C_2\geq2$. The
Jensen or Cauchy-Schwarz inequality for $P_t$ gives $C_2\leq C_1^2$, however,
the exact values of $C_1$ and $C_2$ are not known to the authors knowledge.
The gradient bound \eqref{HQLi} is far more useful than \eqref{DM}, and has
for instance many consequences in terms of functional inequalities for $P_t$,
including Poincar\'e inequalities, Gross logarithmic Sobolev inequalities,
Cheeger type inequalities, and Bobkov type inequalities, as presented in
section \ref{se:inegs}. As we shall see later, \eqref{HQLi} is much harder to
obtain than \eqref{DM}.

More generally, one may consider for any $p\geq1$ and $t\geq0$ the best
constant $C_p(t)$ in $[0,\infty]$ (i.e. the smallest possible, possibly
infinite) such that
$$
\forall f\in\cCP(\He),~\left|\nabla P_t f\right|^p\leq C_p(t)P_t(\left|\nabla
 f\right|^p).
$$
It is immediate that $C_p(0)=1$. According to \cite{DriverMelcher} and
\cite{HQLi}, for every $p\geq1$ and $t>0$, the quantity $C_p(t)$ belongs to
$(1,\infty)$ and does not depend on $t$. In particular, $C_p$ is discontinuous
at $t=0$, and this reflects the fact that the $\Gad$ curvature of $L$ is
$-\infty$.

The aim of this paper is mainly to provide simpler proofs of the gradient
bounds \eqref{DM} and \eqref{HQLi}. We also give in section \ref{se:inegs} a
collection of consequences of \eqref{HQLi} in terms of functional inequalities
for the heat kernel of the Heisenberg group. Section \ref{se:facts} gathers
some elementary properties of the Heisenberg group used elsewhere. Section
\ref{se:oip} provides a direct simple proof of the reverse Poincar\'e
inequality \eqref{eq:RPI-from-DM} without using \eqref{DM} or \eqref{HQLi}.
Sections \ref{se:DM} and \ref{se:HQLi} provide elementary proofs of \eqref{DM}
and \eqref{HQLi} respectively.
 
\section{Elementary properties of the Heisenberg group}
\label{se:facts}

We summarize in this section the main properties of the Heisenberg group that
we use in the present paper. For more details on the geometric aspects, we
refer to \cite{gromov,montgomery,juillet}. The link with the Brownian motion
and its L\'evy area is considered for instance in \cite{baudoin} and
\cite{neuenschwander}. From now on, we shall use the notations
$$
\lag f\rag \quad\text{and}\quad \lag f,g\rag =\lag fg \rag
$$
to denote the integral of a function $f$ with respect to the Lebesgue measure
in $\R^{3}$ and the scalar product of two functions $f,g$ in
$\mathrm{L}^{2}(\R^{3},\R)$. The Heisenberg group $\He$ is the set of matrices
$$
M(x,y,z)= \begin{pmatrix} 1&x&z\\0&1&y\\0&0&1\end{pmatrix}
$$
equipped with the following non-commutative product
$$
M(x,y,z)M(x',y',z')= M(x+x',y+y', z+z'+xy').
$$ 
The inverse of $M(x,y,z)$ is $M(-x,-y, xy-z)$. It is often more convenient to
work with the Lie algebra of the group. For that, we define
$$
X=\begin{pmatrix} 0&1&0\\0&0&0\\0&0&0\end{pmatrix},~
Y=\begin{pmatrix} 0&0&0\\0&0&1\\0&0&0\end{pmatrix},~
Z= \begin{pmatrix} 0&0&1\\0&0&0\\0&0&0\end{pmatrix}
$$ 
and we consider
$$ 
N(x,y,z)= \exp(xX+yY+zZ),
$$ 
which gives
$$
N(x,y,z)=M\left(x,y,z+\frac{xy}{2}\right).
$$
We shall therefore identify a point $\Bx=(x,y,z)$ in $\R^{3}$ with the matrix
$N(x,y,z)$ and endow $\R^{3}$ with this group structure that we denote $\Bx
\bu\By$ which is
$$
(x,y,z)\bu (x',y',z')= (x+x',y+y', z+z'+\frac{1}{2}(xy'-yx')).
$$
The left invariant vector fields which are given by
$$
X(f)= \lim_{\epsilon\to 0} \frac{f(\Bx\bu (\epsilon,0,0))- 
f(\Bx)}{\epsilon}= (\partial_{x}- \frac{y}{2}\partial_{z}) (f),
$$

$$
Y(f)= \lim_{\epsilon\to 0} \frac{f(\Bx\bu (0,\epsilon,0))- 
f(\Bx)}{\epsilon}= (\partial_{y}+ \frac{x}{2}\partial_{z}) (f),
$$

$$
Z(f)= \lim_{\epsilon\to 0} \frac{f(\Bx\bu (0,0,\epsilon))- 
f(\Bx)}{\epsilon}= \partial_{z}(f),
$$
while the right invariant vector fields are given by
$$
\hat X(f)= \lim_{\epsilon\to 0} \frac{f((\epsilon,0,0)\bu \Bx)- 
f(\Bx)}{\epsilon}= (\partial_{x}+ \frac{y}{2}\partial_{z}) (f),
$$
$$
\hat Y(f)= \lim_{\epsilon\to 0} \frac{f((0,\epsilon,0)\bu \Bx)- 
f(\Bx)}{\epsilon}= (\partial_{y}- \frac{x}{2}\partial_{z}) (f).
$$
And $\hat Z= Z$ since for the points on the $z$ axis, left and right multiplications
coincide. The Lie algebra structure is described by the identities $[X,Y]=Z$
and $[X,Z]=[Y,Z]=0$. In what follows, we are mainly interested in the operator
$L=X^{2}+Y^{2}$, and the associated heat semigroup
$(P_{t})_{t\geq0}=(e^{tL})_{t\geq0}$. We shall make a strong use of symmetries
in what follows. They are described by the Lie algebra of the vector fields
which commute with $L$. This Lie algebra is $4$-dimensional and is generated
by the vector fields $\hat X$, $\hat Y$, $Z$, and $\Theta=
x\partial_{y}-y\partial_{x}$. The first ones, which correspond to the right
action, commute with $X,Y,Z$ (as it is the case on any Lie group), the last
one reflects the rotational invariance of $L$. There is also another vector
field which plays an important role : the dilation operator $D$, described by
$$D= \frac{1}{2}(x\partial_{x}+y\partial_{y}) + z\partial_{z}.$$
This operator $D$ satisfies 
\begin{equation} \label{Dil1}
[L,D]=L.
\end{equation} 
Let $(T_{t})_{t\geq0}= (e^{tD})_{t\geq0}$ be the group of dilations generated
by $D$, that is
$$
T_t f(x,y,z)= f(e^{t/2}x, e^{t/2}y, e^tz).
$$
>From the commutation relations, one deduces
\begin{equation} \label{Dil2}
 P_{t}T_{s}=T_{s}P_{e^{s}t},
\end{equation}
and
\begin{equation} \label{Dil3}
 P_{t}D= DP_{t}+ tP_{t}L.
\end{equation} 
Since $P_t$ commutes with left translations, if $l_{\Bx}(f)(\By)= f(\Bx\By)$,
then
$$
P_{t}(f)(\Bx)= l_{x}P_{t}(f)(0)= P_{t}(l_{x}f)(0).
$$
Here we have just formalized the fact that $(P_{t})_{t\geq0}$ is a heat
semigroup on a group. Moreover, since $0$ is a fixed point of the dilation
group, one has from equation \eqref{Dil2}
$$
P_{t}(f)(0)= P_{1}(T_{\log t}f)(0).
$$
This explains why $P_{1}(f)(0)$ gives the whole $(P_t f)_{t\geq0}$. It is well
known that
$$P_{1}(f)(0)=\int_{\R^{3}} f(\By) h(\By) d\By,$$ where $d\By$ is the 
Lebesgue measure on $\R^{3}$ and the function $h$ has the following Fourier
representation in the $z$ variable
\begin{equation}\label{eq:h}
h(x,y,z)= \frac{1}{8\pi^2}\int_{-\infty}^{+\infty}\!e^{i\lambda 
 z}\exp\left(-\frac{r^{2}}{4}\lambda\coth \lambda\right)\frac{\lambda}{\sinh
 \lambda}d\lambda.
\end{equation}
where $r^{2}= x^{2}+y^{2}$ if $\Bx=(x,y,z)$. This formula appeared
independently in the works of Gaveau and L\'evy. It is not easy to deduce from
this formula any good estimates on $h$, and it is not even easy to see that
$h$ is positive. Nevertheless, there are quite precise bounds on this function
and its derivatives, see for instance \cite{Gaveau,Gaveau2} and
\cite{HQLi,Li2}, which may be expressed in terms of the Carnot-Carath\'eodory
distance. The Carnot-Carath\'eodory distance may be defined as
$$ 
d(\Bx,\By)= \sup_{\{f\text{ such that }\,\Ga(f,f) \leq1\}}f(x)-f(y)
$$
where $\Ga$ is as in \eqref{eq:GaH}. Here, the explicit form is not easy to
write, but we shall only need to express the distance from $0$ to $\Bx$. In
order to do that, it is better to describe the constant speed geodesics
starting from $0$. First, straight lines in the $(x,y)$ plane passing through
$(0,0)$ are geodesics, and the other ones are helices, whose orthogonal
projection on the horizontal plane $\{z=0\}$ is a circle containing the
origin (see \cite{baudoin}). If a point $\Bx(s)$ is moving at constant speed (for the
Carnot-Carath\'eodory distance on the geodesic, its horizontal projection
$\Bp(s)$ moves with unit speed (for the Euclidean distance) on this circle.

Moreover, the height $z(t)$ of the point $\Bx(s)$ is the surface between the
segment $(0, \Bp(s))$ and the circle, which comes from the fact that for any
curve in $\R^{3}$ whose tangent vector is a linear combination of $X$ and $Y$,
one has
$$
dz= \frac{1}{2}(xdy-ydx),
$$ 
that is the area spanned by the point $\Bp(s)$ in the plane. Note that all the
geodesics end on the $z$ axis, which is therefore the cut-locus of the point
$0$. Those geodesics may be parameterized as follows, using a complex notation
for the horizontal projection $\Bp(s)$

\begin{equation}\label{eq:geodesics}
 \Bp(s) = u\left(1- \exp\left(\frac{is}{\left|u\right|}\right)\right) %
 \quad\text{and}\quad %
 z(s)= \frac{\left|u\right|^{2}}{2}%
 \left(\frac{s}{\left|u\right|}%
   -\sin\left(\frac{s}{\left|u\right|}\right)\right).
\end{equation}
Here, $u$ is the center of the circle which is the horizontal projection of
the geodesic, and $s$ is distance from $0$ ($s\in (0, 2\pi\left|u\right|)$).
When $\left|u\right|$ goes to infinity, we recover the straight lines.

If we call $d(\Bx)$ the Carnot-Carath\'eodory distance from $0$ to $\Bx$, it
is easy to see from that if $\Bx= (x,y,z)$, and $s= d(\Bx)$, then
$$
d\left(\frac{x}{s}, \frac{y}{s}, \frac{z}{s^{2}}\right)= 1.
$$ 
This corresponds to a change of $u$ into $\frac{u}{s}$. Now, the unit ball for
the Heisenberg metric is between two balls for the usual Riemannian metric of
$\R^{3}$, and therefore one concludes easily that the ratio
$$\frac{(x^{2}+y^{2})^{2}+ z^{2}}{d^{4}(x,y,z)}$$ is bounded above 
and below. Although, if $R= ((x^{2}+y^{2})^{2}+ z^{2})^{1/4}$, the function
$\Ga(R,R)$ is bounded above but not below.

Let $h$ be the heat kernel density at time $1$ and at the origin, given by
\eqref{eq:h}. The main properties of $h$ used in the present paper are the
following. Here for real valued functions $a$ and $b$, we use the notation
$a\simeq b$ when the ratio $a/b$ is bounded above and below by some positive
constants. First,
\begin{equation} \label{esth1}
 h(\Bx)\simeq
 \frac{\exp(-\frac{d^{2}(\Bx)}{4})}{\sqrt{1+
     \left\|\Bx\right\|d(\Bx)}},
\end{equation} 
where $\left\|\Bx\right\|$ denotes the Euclidean norm of the projection of
$\Bx$ onto the plane $\{z=0\}$. Then, for some constant $C$
\begin{equation} \label{esth2} 
\Ga(\log h,\log h)(\Bx) \leq C (1+d(\Bx))
\end{equation}
and
\begin{equation} \label{esth3} 
\left|Z(\log h)\right| \leq C.
\end{equation}
The last one is not completely explicited in \cite{HQLi} but follows easily
from the estimation of $W_{1}$ page 376 of this paper.

\begin{lem}\label{le:A-Algebra}
 The Schwartz space $\mathcal{S}$ of smooth rapidly decreasing functions on
 the Heisenberg group $\He$ is left globally stable by $L$ and by $P_t$ for
 any $t\geq0$.
\end{lem}

\begin{proof}
 If $ R=(x^2+ y^2)^2 + z^2$, then an elementary computation shows that for
 any positive integer $q$, there exists a real constant $B_q>0$ such that
 $L((1+R)^{-q}) \leq B_q(1+R)^{-q}$. As a consequence, $P_t ((1+R)^{-q}) \leq
 e^{B_q t} (1+R)^{-q}$. We may see the class $\mathcal{S}$ as the class of
 smooth functions such that for any non negative integers $a,b,c$ and any
 positive integer $q$, the function $\hat X ^a\hat Y^b Z^c (f)$ is bounded
 above by $(1+R)^{-q}$. From that and the above, it is clear that if $f$ is
 in $\mathcal{S}$, such is $P_t f$. On the other hand, the stability of
 $\mathcal{S}$ by $L$ is straightforward.
\end{proof}

\section{Reverse Poincar\'e inequalities}
\label{se:oip}  

We show here how to deduce a reverse Poincar\'e inequality as
\eqref{eq:RPI-from-DM}. The method is simple and direct, and does not rely on
a gradient bound such as \eqref{DM} or \eqref{HQLi}.

\begin{thm}[Reverse local Poincar\'e inequality]\label{th:RSG}
 For any $t\geq0$ and any $f\in\cC(\He)$,
 $$
 t\Ga(P_{t}f,P_{t}f) \leq P_{t}(f^{2})-(P_{t}f)^{2}.
 $$  
\end{thm}

\begin{proof}
 Since we work on a group, it is enough to prove this for $\Bx=0$. Then,
 thanks to the dilation properties, it is enough to prove it for $t=1$. Now,
 consider the vector field $\hat X$, which coincides with $X$ at $\Bx=0$, and
 let as before $h$ be the density of the heat kernel for $t=1$ and $\Bx=0$
 given by \eqref{eq:h}. We want to bound, for a smooth compactly supported
 function $f$
 $$
 (\hat X P_{1}f)^{2}= (P_{1}(\hat X f))^{2}= \lag f, \hat X h\rag^{2}
 $$
 where the last identity comes from integration by parts. The first remark is
 that we may suppose that $\lag f h\rag=0$ since we may always add any
 constant to $f$. We then use Cauchy-Schwartz inequality under the measure
 $h(\Bx) d\Bx$ to get
 $$
 (\hat X P_{1}f)^{2}\leq P_{1}(f^{2})\lag \hat X(\log h)^{2}h\rag.
 $$ 
 Using the same method for $Y$, we get
 $$
 \Ga(P_{1}f, P_{1}f) \leq (P_{1}(f^{2})-(P_{1}f)^{2}) \lag \hat 
 \Ga(\log h,\log h) h\rag
 $$
 where
 $$
 \hat \Ga(u,u)= \hat X(u)^{2}+ \hat Y (u) ^{2}.
 $$
 Now, the rotational invariance of $h$, which comes from $[\Theta,L]=0$,
 shows that $\hat \Ga(\log h)= \Ga(\log h)$, and gives a reverse local
 Poincar\'e inequality with the constant $C= \lag \Ga(\log h) h\rag$. It
 remains to compute this constant. For that, we use the dilation operator $D$
 and the formula \eqref{Dil3}. In $0$, we have $Df=0$, and therefore it reads
 for $t=1$, for any $f$,
 $$
 P_{1}((L-D)f)=0,
 $$ which means that $h$ is the invariant measure 
 for the operator $L-D$, or in other words that
 $$
 (L+D+2)h=0
 $$
 since the adjoint of $D$ is $-D-2$. Multiply both sides by $\log h$ and using
 integration by parts (can be rigorously justified by using estimates on $h$)
 gives
 $$
 \lag \log h, Lh\rag = -\lag \Ga(\log h ,\log h) ,h\rag.
 $$
 Moreover, we have
 $$
 \lag \log h ,(D+2)h\rag %
 = -\lag h, D\log h\rag %
 = -\lag D h\rag= 2\lag h \rag %
 =2
 $$
 and therefore
 $$
 \lag \Ga(\log h, \log h) ,h\rag = 2.
 $$  
 If we had done the same reasoning on $\R^{n}$ with the usual Laplace
 operator, and the corresponding dilation operator, we would have find a
 reverse Poincar\'e inequality with constant $n/2$ instead of $\frac{1}{2}$.
 The reason comes from symmetry properties and the same will allow us to
 divide the constant by 2 in the Heisenberg case. In fact, because of the
 rotational invariance of $h$, we have, for any vector $(a,b)\in\R^{2}$ with
 $a^2+b^2=1$,
 $$
 \lag (a\hat X(\log h)+ b\hat Y(\log h))^{2},h\rag %
 = \lag (\hat X(\log h))^{2},h\rag %
 = \frac{1}{2}\lag \hat \Ga(\log h, \log h) ,h\rag %
 = 1.
 $$
 %Here we used the fact that $(a\hat X(\log h)+ b\hat Y(\log h))^{2}$ is
 %equal to $(\hat X(\log h))^2$ taken at some other point.
 Now, we may write
 $$
 \Ga(P_{1}f, P_{1}f)(0) %
 = \sup_{a^2+b^2=1} %
 (a\hat X P_{1}(f) + b\hat Y P_{1}(f))^2
 $$
 and use the same Cauchy-Schwarz inequality to improve the bound to
 $$
 \Ga(P_{1}f,P_{1}f) \leq P_{1}(f^{2})-(P_{1}f)^{2}.
 $$
\end{proof}

\begin{rmk}[Optimal constants]
 Equality is achieved in theorem \eqref{th:RSG} when $f = \hat X(\log h)$ for
 instance. To see it, note that by symmetry, $\lag (\hat X\log h)^2 h\rag =
 \lag (\hat Y\log h)^2 h\rag = 1$. By the rotational invariance of $h$, we
 get more generally that $\lag (a\hat X\log h+ b\hat Y\log h)^2 h\rag =1$ for
 any $a^2+ b^2=1$. As a consequence, $\lag \hat X(\log h)\hat Y(\log
 h),h\rag =0$. For $f = \hat X(\log h)$, this gives $XP_1f (0)= \lag \hat X
 f, h\rag= -\lag f, \hat X( \log h) h \rag = -1$ and $YP_1 f =0$, which is
 the desired equality. Note the difference with the elliptic case: for the
 heat semigroup $(P_t)_{t\geq0}$ in $\R^{n}$ or for any manifold with non
 negative Ricci curvature, one has for every $t\geq0$ and any smooth $f$,
 $$
 2t\Ga(P_{t} f,P_{t}f) \leq P_{t}(f^{2})-(P_{t}f)^{2}.
 $$
 Note also, as we already mentioned in the introduction, that the H.-Q. Li
 gradient bound \eqref{HQLi} provides simply by semigroup interpolation a
 result similar to theorem \ref{th:RSG}, with a constant $2tC_1^{-2}$ instead
 of $t$. These two constants are equal if and only if $C_1=\sqrt{2}$. At the
 level of reverse local Poincar\'e inequalities, a necessary and sufficient
 condition for the efficiency of the semigroup interpolation technique is
 that $C_1=\sqrt{2}$. Similarly, by using the Drivier and Melcher gradient
 bound \eqref{DM} we obtain the condition $C_2=2$. It is thus tempting to
 conjecture that $C_1^2=C_2=2$.
\end{rmk}

\begin{rmk}[Carnot groups]
 In the class of nilpotent groups, there is an interesting subclass, which
 are the Carnot groups, that is the nilpotent groups with dilations, see \cite{baudoin}, 
 \cite{gromov}. Let $(X_{1},\ldots,X_{n_{0}})$ the generators at the first
 level of such a Carnot group, and $L= X_{1}^{2}+ \cdots+ X_{n_{0}}^{2}$. In
 those groups again there is a dilation operator $D$ such that $[L,D]=L$ and
 $D ^{*}= -D-\frac{n}{2}Id$. The parameter $n$ is what is called the
 homogeneous dimension of the group. The same applies in this case, except
 that we no longer have always enough rotations to insure that $\Ga(h,h)=
 \hat \Ga(h,h)$ where the hat corresponds to the ``chiral'' action. But we
 may replace this argument by the fact that $\hat P_{t}(\cdot)(0)=
 P_{t}(\cdot)(0)$ where $(\hat P_{t})_{t\geq0} = (\exp(t \hat L))_{t\geq0}$
 and we would get as a bound
 $$
 \Ga(P_{t}f,P_{t}f) \leq \frac{n}{2n_{0}t} (P_{t}(f^{2})-(P_{t}f)^{2})
 $$
 recovering at the same time the Heisenberg and Euclidean cases. For the
 $(2p+1)$-dimensional Heisenberg group $\He_{2p+1}$ we have $n_{0}= 2p$ while
 the homogeneous dimension is $2p+2$, and therefore here the inequality
 writes
 $$
 \Ga(P_{t}f,P_{t}f) \leq \frac{p+1}{2pt}( P_{t}(f^{2})-(P_{t}f)^{2})
 $$ 
 and the constant approaches the Euclidean one when $p$ goes to infinity.
\end{rmk}

\section{A proof of the Driver-Melcher inequality}
\label{se:DM}

We give here an elementary proof of the Driver and Melcher gradient bound
\eqref{DM}. The argument is simply an integration by parts followed by the
upper bound on $\Ga(\log h, \log h)$ obtained in section \ref{se:oip}. Indeed,
from the inequalities \eqref{esth1} and \eqref{esth2}, it is quite clear that
the constant
\begin{equation} \label{defA} %
 A= \int \Vert\Bx\Vert\Ga(\log h, \log h)(\Bx) h(\Bx) d\Bx
\end{equation} 
is finite, where $\Vert\Bx\Vert$ denotes as usual the Euclidean norm of the
horizontal projection of the point $\Bx$. Then, we have the following theorem.

\begin{thm}\label{th:DMA}
 With the constant $A$ defined by \eqref{defA}, we have for every $t\geq0$
 and $f\in\cC(\He)$,
 $$
 \Ga(P_{t} f, P_{t}f) \leq 2(A+4) P_{t}(\Ga(f,f)).
 $$
\end{thm}

\begin{proof} 
 We assume that $\Bx =0$ (by group action) and $t=1$ (by dilation). Then, we
 write
 $$
 XP_{1}f(0)= P_{1}(\hat X f)(0)= \lag (X+yZ)f h\rag.
 $$
 An integration by parts for $\lag yZ(f),h\rag= \lag y(XY-YX)(f) ,h\rag$
 gives
 $$
 \lag X(f),(yY(\log h) + 1) h\rag -\lag Y(f), yX(\log h) h\rag
 $$
 and a similar formula holds for $YP_{1}f$. Next, we take a vector
 $(a,b)\in\R^{2}$ of unit norm and we use the Cauchy-Schwarz inequality to
 get
 $$
 (aXP_{1}(f)(0)+ bYP_{1}(f)(0))^{2} %
 \leq P_{1}(X(f)^{2})A_{1}+ P_{1}(Y(f)^{2})A_{2}
 $$
 where 
 $$
 A_{1}= P_{1}[((yY(\log h)+2)a-xY(\log h)b)^{2}]
 $$ 
 and 
 $$
 A_{2}= P_{1}[((xX(\log h)+2)b-yY(\log h)a)^{2}].
 $$
 The desired inequality comes then from the upper bound
 $$
 \max(A_{1}, A_{2})\leq A_{1}+A_{2} \leq 2(A+4).
 $$
 Note that the obtained constant $2(A+4)$ is certainly not the optimal one.
\end{proof}

\begin{rmk}[Counter example]
 Unlike the elliptic case, the reverse local Poincar\'e
 \eqref{eq:RPI-from-DM} and the local Poincar\'e \eqref{eq:PI-from-DM}
 inequalities are not in general equivalent. A simple example is provided on
 $\R^{2}$ with the operator
 $$
 L = \partial_{x} ^{2} + x \partial_{y}
 $$
 for which the corresponding diffusion process starting from $(x,y)$ is up to
 some constant
 $$
 U_{t} = \left(x+B_{t},\ y+tx+ \int_{0}^{t}\!\!B_{s} ds\right)
 $$
 where $(B_{t})_{t\geq0}$ is a Brownian motion on $\R$.
 % ($2^{-1/2}$ times the standard one).
 In this example, the heat kernel is Gaussian and the semigroup
 $(P_{t})_{t\geq0}$ is quite easy to compute, while $\Ga(f,f) = (\partial_{x}
 f)^{2}$. In this situation, it is easy to see, using Cauchy-Schwarz
 inequality, that
 $$
 tC\Ga(P_{t} f,P_{t} f) \leq P_{t}(f^{2})-(P_{t}f)^{2}
 $$  
 for some constant $C<2$, while the inequality
 $$
 P_{t} f^{2} -(P_{t}f)^{2} \leq C(t) P_{t}\Ga(f,f)
 $$
 does not hold for any constant $C(t)$, as one may see with a function $f$
 depending on $y$ only. For example, with $f(x,y)=y$, one has $P_{t}f =
 y+tx$, which depends on the variable $x$. This kind of hypoelliptic
 situation differs strongly from the case of the Heisenberg group, since here
 $\Ga(f,f)=0$ does not imply that $f$ is constant.
 %
 % From db 2007-10-17
 % 
 % Gradient bounds for $L=\partial_{x}^{2}+ y\partial_{x}$.
 %
 % I shall get $(XP_{t} f)^{2} \leq \frac{2}{t}(P_{t}f^{2}-(P_{t}f^{2})$
 % from commutation relations.
 %
 % Let set $X= \partial_{x}$ and $Y= \partial_{y}$. We have $[L,X]= -Y$,
 % $[L,Y]=0$.
 % Then we get $P_{t }X = XP_{t}-tYP_{t}$. From this, we write
 %
 % $$P_{t}f^{2}-(P_{t}tf)^{2} \geq 2 \int_{0}^{t}(P_{s}XP_{t-s}f)^{2}
 % ds= tA^{2}-t^{2}AB + \frac{t^{3}}{3}B^{2}$$ with
 % $A= XP_{t}f$, $B= YP_{t}f$.
 %
 % Then
 % $$tA^{2}-t^{2}AB + \frac{t^{3}}{3}B^{2} \geq tA_{2}/4,$$ and we are
 % done.
 %
 % We also get this way
 %
 % $$P_{t}f^{2}-(P_{t}tf)^{2}\geq \frac{1}{6t^{3}}(YP_{t}f)^{2}.$$
 %
 % I did not check precisely, but I am ready to bet that these are the
 % sharp constants.
 %
 % I tried to use a similar thing with Heisenberg, but failed. 
\end{rmk}

\section{Two proofs of the H.-Q. Li inequality}
\label{se:HQLi}

In this section, we propose two alternate and independent proofs of the H.-Q. Li inequality
\eqref{HQLi}.  The first proof uses some basic symmetry considerations and  a
particular case of the Cheeger inequality of theorem \ref{th: cheeger2} that we have to show by hands.
 The second proof relies on an explicit commutation between the complex gradient and the heat semigroup. Both mainly rely on the previous sharp estimates on the heat kernel that were obtained in \cite{Gaveau2}.

\subsection{Via a Cheeger type inequality}

\begin{lem}\label{le:cheeger}
 For any real $R>0$, there exists a real constant $C>0$ such that for any
 smooth $f:\He\to\R$ which vanishes on the ball centered at $0$ and of radius
 $R$ for the Carnot-Carath\'eodory distance, we have
 $$
 \int \left|f\right| h d\Bx \leq C \int \left|\nabla f\right| h 
 d\Bx
 $$ 
 where $h$ is as before the density of $P_{1}(0, d\Bx)$.
\end{lem}

\begin{proof}
  One may safely assume that $R=1$ by a simple scaling. Next, we make use of
  the polar coordinates which appear in (\ref{eq:geodesics}). Namely, we
  parameterize the exterior of the unit ball by $(u, s)$, with $u\in \C,
  ~\left|u\right|\geq \frac{1}{2\pi}$ and $s\in (1, 2\pi\left|u\right|)$, with
  \begin{equation}
    (x+iy, z) %
    = \left( u\left(1- \exp\left(\frac{is}{\left|u\right|}\right)\right),
      \frac{\left|u\right|^{2}}{2}\left(\frac{s}{\left|u\right|}-\sin\left(\frac{s}{\left|u\right|}\right)\right)\right).
  \end{equation}
  The unit ball is the set $\{s\leq 1\}$, and since $f$ is supported outside
  the unit ball, we write
  $$
  \left|f(u,s)\right| %
  = \left|\int_{1}^{s} \nabla f(u,t)\cdot e_{t} dt \right| %
  \leq \int_{1}^{s} \left|\nabla f\right|(u,t) dt
  $$
  where $e_{t}$ is the unit vector along the geodesic. Let us write $A(u,t) du
  dt$ the Lebesgue measure on $\R^{3}$ in those coordinates (we shall see the
  precise formula below). We write
  $$
  \int \left|f(u,s)\right|h(u,s) A(u,s) du ds %
  \leq \int \left|\nabla f\right|(u,t)\left(\int_{t}^{2\pi \left|u\right|}A(u,s)
    h(u,s) ds\right) du dt
  $$
  and we shall have proved our inequality when we have proved that
  $$
  \int _{t}^{2\pi\left|u\right|} A(u,s)h(u,s) ds \leq C A(u,t)h(u,t),
  $$
  for any $(u,t)$ such that $\left|u\right|\geq \frac{1}{2\pi}$ and $t\geq 1$.
  In this computation, we forget the points in the $(x,y)$ plane and the
  $z$-axis, but this is irrelevant since they have $0$-measure. The
  computation of the Jacobian gives
  $$
  A(u,s)= 16\left|u\right|\sin
  \left(\frac{s}{2\left|u\right|}\right)\left(\frac{s}{2\left|u\right|}-\sin\left(\frac{s}{2\left|u\right|}\right)\right)
  $$
  and the estimate (\ref{esth1}) shows that we may replace $h(u,s)$ by
  $$
  \frac{\exp(-\frac{s^{2}}{4})}{\sqrt{1+2s\left|u\right|\sin(\frac{s}{2\left|u\right|})}}
  $$ 
  since the Euclidean norm of the horizontal projection of the point whose
  coordinates are $(u,s)$ is $2\left|u\right|\sin(\frac{s}{2\left|u\right|})$.
  Setting $\tau= \frac{s}{2\left|u\right|}$ and $r= \left|u\right|$, the
  question is therefore to check that, for some constant $C $, for any $r\geq
  \frac{1}{2\pi}$, for any $\tau_{0}\geq \frac{1}{2r}$, one has
  $$
  r\int_{\tau_{0}}^{\pi} \frac{\sin\tau(\tau-\sin\tau)}{\sqrt{1+4r^{2}\tau\sin
      \tau}}e^{-\tau^{2}r^{2}} d\tau %
  \leq C\frac{\sin\tau_{0}(\tau_{0}-\sin\tau_{0})}{\sqrt{1+4r^{2}\tau_{0}\sin
      \tau_{0}}}e^{-\tau_{0}^{2}r^{2}}.
  $$
  Up to some constant, we may replace $\sin\tau(\tau-\sin \tau)$ by $\tau^{4}$
  on $(0, \frac{\pi}{2})$ and by $\pi-\tau$ on $(\frac{\pi}{2}, \pi)$. In the
  same way, we may replace $\sqrt{1+4r^{2}\tau\sin \tau}$ by $r\tau$ when
  $\tau< \frac{\pi}{2}$ (since $r\tau \geq\frac{ 1}{2}$) and by $1+
  r\sqrt{\pi-\tau}$ when $\tau\in (\frac{\pi}{2}, \pi)$.
  
  We first consider the case where $\tau_{0}<\frac{\pi}{2}$, and divide the
  integral into $\int_{\tau_{0}}^{\pi/2}$ and $\int_{\pi/2}^{\pi}$. Using the
  above estimates, these integrals can be bounded by the correct term by using
  the fact that
  $$
  \int_{A}^{\infty} s^{p}\exp(-s^{2}) ds \leq C_{p} A^{p-1}\exp(-A^{2}).
  $$ 
  When $\tau_{0}> \frac{\pi}{2}$, one uses the same estimates, bounding above
  $\pi-\tau$ by $\pi-\tau_{0}$ and $(1+ r\sqrt{\pi-\tau})^{-1}$ by 1 in the
  integral, and using the fact that $r$ is bounded below on our domain.
  
  Observe that the same reasoning on a ball of radius $\epsilon$ would provide
  a constant which goes to infinity when $\epsilon$ goes to $0$, as for the
  usual heat kernel on $\R^{d}$.
\end{proof}

In fact, we shall also use a slightly improved version of lemma
\ref{le:cheeger}.

\begin{lem}\label{le:cheeger-bis}
 For every real $R>0$, if $B$ is the ball centered at $0$ and of radius $R$
 for the Carnot-Carath\'eodory distance, there exists a real constant $C>0$
 such that for any smooth $f:\He\to\R$,
 $$
 \int_{B^c}\left|f-m\right|\,hd\Bx %
 \leq C \int \left|\nabla f\right| h d\Bx
 $$ 
 where $B^c=\He\setminus B$ is the complement of $B$, where
 $m=|B|^{-1}\int_B\!f(x)\,d\Bx$ is the mean of $f$ on $B$, and where $h$ is as
 before the density of $P_{1}(0, d\Bx)$.
\end{lem}

For proving this last lemma, we will need the following $L^1$-Poincar\'e, also called $(1,1)$ Poincar\'e, on balls. This inequality  
can be in fact thought of as a Cheeger type inequality on balls. See 
\cite{Maheux} and references therein. This last lemma shall also be 
used in the next section where we prove the H.-Q. Li inequality via 
complex analysis.

\begin{lem}\label{le:cheloc}
  For any real $R>0$, if $B$ denotes the ball centered at $0$ and of radius
  $R$ for the Carnot-Carath\'eodory distance, there exists a real constant
  $C>0$ such that for any smooth $f:\He\to\mathbb{R}$, by denoting
  $m=|B|^{-1}\int_B\!f(x)\,d\Bx$ the mean of $f$ on $B$,
  $$
  \int_B\!\left|f(x)-m\right|\,d\Bx%
  \leq C\,\int_B\,\left|\nabla f\right|(x)\,d\Bx.
  $$
\end{lem}

We can now make the proof of lemma \ref{le:cheeger-bis}.

\begin{proof}[Proof of lemma \ref{le:cheeger-bis}]
  As in lemma \ref{le:cheeger}, we may safely assume that $R=1$ by a simple
  scaling. For any auxillary function $g:\He\to\R$, we have by denoting
  $m=|B|^{-1}\int_B\!f\,d\Bx$,
  $$
  \int_{B^c}\left|f-m\right|\,h\,d\Bx %
  \leq \int\!\left|f-g\right|\,h\,d\Bx %
  +\int_{B^c}\!\left|g-m\right|\,h\,d\Bx.  
  $$
  Now we choose $g$ such that $g(\xi,s)=f(\xi,\min(s,1))$ where $(\xi,s)$
  denotes the polar coordinates in $\He$ as in the proof of lemma
  \ref{le:cheeger}. More precisely, $\xi\in\partial B$ and the set $\{s\leq
  1\}$ is the unit ball. For the first term the desired gradient bound follows
  then by elementary arguments as in lemma \ref{le:cheeger}. For the second
  term, we write
  $$|f(\xi,1) - m| \leq \int_{s=0}^1 \left( |f(\xi,1) - f(\xi,s)| + |f(\xi,s)-m| \right) \frac{A(\xi,s) ds}{C(\xi)}
  $$ where $C(\xi)=\int_{s=0}^1 A(\xi,s) ds$. We can now conclude by using elementary arguments similar as before and the $L^1-$Poincar\'e inequality of lemma \ref{le:cheloc}.
\end{proof}

Note that lemma \ref{le:cheeger} can be deduced directly from lemma
\ref{le:cheeger-bis}.
% The constant $C$ which appears in lemma \ref{le:cheeger-bis} depends on $R$.
We are now in position to prove the H.-Q. Li inequality \eqref{HQLi}.

\begin{proof}[Proof of \eqref{HQLi}]
 With the help of lemmas \ref{le:cheeger} and \ref{le:cheeger-bis}, we may
 reduce the study of the H.-Q. Li inequality to functions which are
 \begin{itemize}
 \item either supported in a ball of radius 1 for the Carath\'eodory 
   metric;
 \item either supported in a cylinder of radius 2 around the $z$ axis (without the unit ball);
 \item either supported outside a cylinder around the $z$-axis.
 \end{itemize}
 Indeed, let see how one may reduce first to the case of a function supported
 either in a ball or outside a ball. If $f$ is any smooth function and $\phi$
 a smooth cutoff function with values $1$ on a ball $B$ of radius $<1$ and
 vanishing outside a ball of radius $1$, we write $f= f\phi+ f(1-\phi) =
 f_{1}+ f_{2}$. Clearly, in order to obtain \eqref{HQLi}, one can add any
 prescribed constant to $f$. In particular, one can assume that
 $\int_B\!f\,d\Bx=0$. Assuming that we know the inequality for $f_{1}$ and
 $f_{2}$, we bound
 $$ 
 \lag \hat X(f), h \rag %
 \leq C\lag (\left|\nabla f_{1}\right|+\left|\nabla f_{2}\right|) ,h \rag
 $$
 then we make use of 
 $$
 \left|\nabla f_{1}\right|+ \left|\nabla f_{2}\right| %
 \leq
 \left|\nabla f\right|+ 2\left|f\right|\left|\nabla \phi\right|
 $$
 and since $\left|\nabla \phi\right|$ is supported outside the unit ball,
 $$\left|f\right|\left|\nabla \phi\right| \leq ||\nabla \phi||_\infty |f| 1_{B ^c}
 $$
 so one
 has by lemma \ref{le:cheeger-bis}
 $$
 \lag \left|f\right|,\left|\nabla \phi\right| h\rag %
 \leq C \lag \left|\nabla f\right| , h\rag.
 $$
 We repeat the same operation with a cutoff function for the neighborhood of
 the $z$-axis. 
 Now, when $f$ is supported inside the ball, we may use the
 method that we used in the proof of theorem \ref{th:DMA}, and the fact that
 $\left|\nabla \log h\right|(\Bx)\leq C d(\Bx)$, which is bounded on the unit
 ball. 
 If $f$ is supported inside the cylinder around the $z$-axis and vanishes on the unit ball, we write,
 with section \ref{se:facts} notations,
 $$
 \lag \hat X(f) ,h\rag = \lag X(f), h\rag + \lag f ,\frac{y}{2}Z(\log h)
 h\rag
 $$ 
 and then we use the fact that $\frac{y}{2}Z(\log h)$ is bounded on the
 cylinder. It remains to observe that
 $$
 \lag \left|f\right|, h\rag  \leq C \lag\left|\nabla f\right|, h\rag
 $$ 
 thanks to lemma \ref{le:cheeger}. 
 It remains to deal with a function which
 is supported outside a cylinder around the $z$-axis. We shall choose another
 integration by parts. For that, let us use a complex notation and write
 $$
 \nabla (f) = X(f)+iY(f) %
 \quad\text{and}\quad %
  \hat \nabla (f) = \hat X(f)- i \hat Y(f).
 $$  
 Note the change of sign in front of $i$ in the second expression. We want to
 bound
 $$
 \lag \hat \nabla (f) , h\rag = -\lag f ,\hat \nabla h \rag.
 $$
 Now, since $h$ is radial, we have
 $$
 \hat \nabla h= \frac{x-iy}{x+iy}\nabla h
 $$
 which comes from the fact that $x\partial_{y} h= y\partial_{x} h$. Let us
 call $\Psi(x,y)=\exp(-2i \theta)$ the function $\frac{x-iy}{x+iy}$, where
 $\theta$ is the angle in the plane $(x,y)$. Then, we integrate again by parts
 and get
 $$
 \lag \hat \nabla f ,h\rag= -\lag f, \Psi(x,y)\nabla h\rag %
 = \lag \nabla f ,\Psi(x,y)h\rag + \lag f ,\nabla (\Psi) h \rag.
 $$
 We then conclude observing that $\Psi$ is bounded and $\left|\nabla
   \Psi\right|$ is bounded outside the cylinder around the $z$ axis. We
 therefore have
 $$
 \left| \lag \hat \nabla (f), h\rag\right|\leq \lag \left|\nabla f\right|,
 h\rag + C \lag \left|f\right|, h\rag
 $$
 and we use again lemma \ref{le:cheeger-bis} to conclude the proof.
\end{proof}

\subsection{Via a complex quasi-commutation}
\label{se:C}

In $\mathbb{R}^n$, it is known that the gradient $\nabla$ commute with the
Laplace operator. This commutation leads to the commutation between $\nabla$
and the heat semigroup $P_t=e^{t\Delta}$ and therefore to the inequality:
\[
\left|\nabla P_t f\right| = \left|P_t \nabla f\right| \le P_t \left|\nabla f\right|.
\] 
In the Heisenberg group, we can follow the same pattern of proof. Nevertheless, several difficulties appear that make the proof quite delicate and technical at certain points. For sake of clarity, before we enter the hearth of the proof, let us precise our strategy. The Lie algebra structure:
\[
[X,Y]=Z,\quad [X,Z]=[Y,Z]=0
\]
leads to the commutation:
\[
(X+iY)L=(L-2iZ)(X+iY),
\]
where $L=X^2+Y^2$. At the level of semigroups, it leads to the \textit{formal} commutation:
\begin{align}\label{formalcommutation}
(X+iY)P_t =e^{t(L-2iZ)}(X+iY)= e^{-2itZ}P_{t}(X+iY).
\end{align}
This commutation is only formal because as we will see the semigroup
associated to the complex operator $L-2iZ$ is not globally well defined. More
precisely, complex solutions to the heat equation $\frac{\partial u}{\partial
 t} =(L-2iZ)u$, $u(0,\cdot)=f$ are represented by a kernel which is 
 nothing esle than the holomorphic complex extension in the $z$ 
 variable of the heat kernel, at the point $z+2it$. Unfortunately, 
 this kernel has poles, and this solution may have singularities. Nevertheless, we will see that
if the initial condition $f$ is a complex gradient, then solutions to this
equation do not explode. More precisely, we may add to this kernel 
any kernel which has no effect on gradients and which cancels the 
poles of the previous extension. Doing this, we shall 
produce  an integral
representation of the solution, without poles. This representation is 
of course not unique.  If we could choose the kernel in such a way that the ratio of it with
the density  $p_t$ is bounded, then the H.-Q. Li inequality would easily
follow. However, we will prove that  such a kernel does not exist.
To overcome this difficulty, we will use two different kernels depending on the
support of the function $f$. By using a partition of the unity as in our
previous proof of H.-Q. Li inequality and lemma \ref{le:cheloc} we
will then be able to conclude.

\

We now enter into the hearth of the proof. In what follows, in order to exploit the
rotational invariance, we shall use the cylindric coordinates $x=r \cos
\theta$, $y=r\sin \theta$ in which the  vector fields $X$ and $Y$ read
\[
X= \cos \theta \partial_{r} - \frac{\sin
\theta}{r} \partial_\theta-\frac{1}{2} r \sin
\theta \partial_{z}
\]
\[
Y= \sin \theta \partial_{r} +\frac{\cos
\theta}{r} \partial_{\theta} +\frac{1}{2} r \cos
\theta \partial_{z}
\]
\[
Z=\partial_{z}.
\]
The heat kernel associated to $(P_t)_{t\geq0}$ writes here in cylindric
coordinates
\begin{equation}\label{gaveau}
p_t(r,z)= \frac{1}{8 \pi^2} \int_{-\infty}^{+\infty} e^{i \lambda z} \frac{\lambda}{\sinh \lambda t}
e^{-\frac{r^2}{4}\lambda  \text{cotanh}  \lambda  t }d \lambda.
\end{equation}
To give a sense to (\ref{formalcommutation}), we begin with the analytic properties of $p_t (r,z)$ in the variable $z$.

\begin{lem}
 Let $t >0$ and $r \ge 0$. The function
 \[
 z \rightarrow p_t (r,z)-\frac{1}{4 \pi^2 \left( t+iz
     +\frac{r^2}{4}\right)^2}-\frac{1}{ 4 \pi^2 \left( t-iz
     +\frac{r^2}{4}\right)^2}
 \]
 admits an analytic continuation on $\left\{ z \in \mathbb{C}, \mid
   \mathbf{Im} z \mid < \frac{r^2}{4}+ 3t \right\}$. The function
 \[
 z \rightarrow p_t (r,z)
 \]
 admits therefore a meromorphic continuation on $\left\{ z \in \mathbb{C},
   \mid \mathbf{Im} z \mid < \frac{r^2}{4}+3 t \right\}$ with double poles at
 $-i \left( t+\frac{r^2}{4} \right)$ and $i \left( t+\frac{r^2}{4} \right)$.
\end{lem}

\begin{proof}
 Let $t >0$ and $r \ge 0$. By using the expression (\ref{gaveau}) for $p_t
 (r,z)$, and
 \[
 \frac{1}{ \left( t+iz+\frac{r^2}{4}\right)^2}
 =\int_{0}^{+\infty} e^{-i \lambda z} e^{-\lambda t}e^{-\lambda\frac{r^2}{4}} \lambda d\lambda,
 \]
 \[
 \frac{1}{ \left( t-iz+\frac{r^2}{4}\right)^2}
 =\int_{0}^{+\infty} e^{i \lambda z} e^{-\lambda t}e^{-\lambda\frac{r^2}{4}} \lambda d\lambda,
 \]
 we obtain
 \begin{align*}
    & p_t (r,z)-\frac{1}{4 \pi^2 \left( t+iz
	+\frac{r^2}{4}\right)^2}-\frac{1}{ 4 \pi^2 \left( t-iz +\frac{r^2}{4}\right)^2} \\
   = &\frac{1}{8 \pi^2} \int_{-\infty}^{+\infty} e^{i \lambda z} \left( \frac{
	e^{-\frac{r^2}{4}\mid \lambda \mid \text{cotanh} \mid \lambda \mid t
	}}{\sinh\mid \lambda \mid t}-2e^{-\frac{1}{4}\mid \lambda \mid
	r^2-\mid \lambda \mid t }\right) \mid \lambda \mid d \lambda
 \end{align*}
 and the desired result follows easily.
\end{proof}

For any $t>0$, $r \ge 0$, and $z \in \mathbb{C}-\{-i(t+\frac{1}{4}r^2)\}$ such
that $\mid \mathbf{Im} z \mid < \frac{r^2}{4}+3t$, let us denote
\[
p_t^* (r,z)=p_t (r,z)-\frac{1}{4 \pi^2 \left( t+iz
+\frac{r^2}{4}\right)^2}.
\]
We have the following commutation property.

\begin{prop}
 If $f: \He \rightarrow \mathbb{R}$ is a smooth function with compact
 support, then
 \[
 (X +i Y) P_t f (0)  = \int_{\He} p^*_t (r,
 z+2it)(X +i Y)f(r, \theta, z) r dr d\theta dz, \quad t > 0.
 \]
\end{prop}

\begin{proof}
 Due to the identities $[X,Y]=Z$ and $[X,Z]=[Y,Z]=0$, we have 
 \[
 (X+iY)L=(L-2iZ)(X+iY).
 \]
 If $f(r,\theta,z)=e^{i\lambda z} g(r,\theta)$, for some $\lambda \in \mathbb{R}$ and some function $g$, we have $Zf=i\lambda f$ and thus
\[
(X+iY)Lf=(L+2\lambda) (X+iY)f.
\]
We deduce therefore,
 \[
 (X +i Y)P_t f (0)=e^{2\lambda t} (P_t (X+iY)f)(0)=e^{2\lambda t}\int_{\He} p_t (r,
 z)((X +i Y)f)(r, \theta, z) r dr d\theta dz.
 \]
 Let us now observe that for
 $t>0$,
 \[
 (X+iY) \frac{1}{\left( t+iz +\frac{r^2}{4}\right)^2}=0
 \]
 and thus
 \[
 (X+iY)p^*_t=(X+iY)p_t.
 \]
Consequently,
\[
 (X +i Y)P_t f (0)=e^{2\lambda t}\int_{\He} p^*_t (r,
 z)((X +i Y)f)(r, \theta, z) r dr d\theta dz.
\]
Now
$$e^{2\lambda t} f(r, \theta, z)= f(r, \theta, z- 2 i t)
$$
and the result for the function $f$ follows by integrating by parts with respect to the variable $z$. For general $f$, we can conclude by using the Fourier transform of $f$ with respect to the variable $z$.
\end{proof}

As a first consequence, we deduce that for every $R>0$, there exists a finite constant $C>0$
such that for every smooth function compactly supported inside a Carnot-Carath\'eodory ball
$\mathbf{B}_R$ of radius $R$,
\[
\left|\nabla 
  P_{1} f\right| (0)\leq C P_{1}( \left|\nabla f\right|)(0).
\]
But of course, here, the constant $C$ that we obtain depends on $R$, and we shall see below that it blows
up when $R \rightarrow +\infty$.

Now, if $R>0$ is big enough, the ball with radius $R$ contains the region of the Heisenberg group whose cylindric coordinates are of the form
$(r=2, \theta \in [0,2\pi], z=0)$ and if $f$ is a smooth function with compact support that vanishes in a ball
with radius $R$,  we have the commutation:
\[
(X +i Y) P_1 f (0)  = \int_{\He} p_1 (r,
z+2i)(X +i Y)f(r, \theta, z) r dr d\theta dz, \quad t > 0.
\]
that follows from the fact that $(X+iY)p_t=(X+iY)p_t^*$ and from the fact that the pole of $(r,z) \rightarrow p_1 (r,z)$ is at $r=2$, $z=0$. The keypoint is then the following estimate:

\begin{prop}\label{estimate_ratio}
There exists $R>0$ such that
\[
\sup_{r^2 +\mid z \mid \ge R} \frac{ \mid p_1 (r, z+2i) \mid}{p_1
(r,z)} <+\infty.
\]
\end{prop}

\begin{proof}
We shall proceed in two steps.

\textbf{Step 1.} We show that for any $\eta >0$,
\[
\sup_{r \ge 3, r^2 \ge \eta \mid z \mid} \frac{\left| p_1 (r, z+2i) \right| }{p_1 (r,z)} < +\infty.
\]
For convenience, and by symmetry, we may assume $z >0$.  Let us first observe that on
our domain:
\begin{align} \label{integrale}
p_1 (r, z+2i)=\frac{1}{8 \pi^2}
\int_{-\infty}^{+\infty} e^{-2\lambda} e^{i \lambda z}
\frac{\lambda}{\sinh \lambda } e^{-\frac{r^2}{4}\lambda
\text{cotanh}  \lambda   }d \lambda
\end{align}
>From \cite{Gaveau2}, it is known that for fixed $r,z$, the function
\[
g:\lambda \rightarrow -i \lambda z+\frac{r^2}{4}\lambda
\text{cotanh} \lambda,
\]
has a unique critical point in the strip $\{ \mid \mathbf{Im}
\lambda \mid < \frac{\pi}{2} \}$. This critical point is $i\theta (r,z)$,
where $\theta (r,z)$ the unique solution in $(0,\frac {\pi} {2})$ of the
equation
\[
\mu (\frac{1}{2}\theta (r,z)) r^2= 4 z,
\]
with $\mu (\theta)=\frac{\theta}{\sin^2 \theta} -\text{ cotan }
\theta$. At this critical point, we have
\[
g( i\theta (r,z)) =\frac{d^2 (r,z)}{4},
\]
where $d(r,z)$ is the Carnot-Carath\'eodory distance from 0 to  the point with cylindric coordinates $(r,\theta, z)$ (this distance does not depend on $\theta$, that is why it is omitted in the notation). In fact, our function $g$
corresponds to $g(r,z,\lambda)= f(\frac{r}{\sqrt 2}, \frac{z}{2}, 2 \lambda)$
where $f$ is the function studied in \cite{Gaveau2}.

Moreover the function $s \rightarrow \mathbf{Re} g(s+i\theta
(r,z))$, grows with $\mid s \mid$, and has a global minimum at
$s=0$, indeed a tedious computation shows that
\begin{align*}
\mathbf{Re}(g(s+i\theta (r,z))- g(i\theta (r,z))) &
=\frac{\sinh^2 2s}{\sinh^2 2s +\sin^2 2\theta(r,z)} (2s \text{ cotanh
} 2s - 2\theta(r,z) \text{ cotan } 2\theta(r,z)) r^2 \\
& \ge \frac{\sinh^2 2s}{\sinh^2 2s +1} (2s \text{ cotanh
} 2s - 1) r^2 \\
& \ge 0.
\end{align*}
Let us finally observe that the previous computation also shows
that there exists $ \delta >
0$ such that for $s \in [-1,1]$
\[
\mathbf{Re} g(s+i\theta(r,z) ) \ge \frac{d^2 (r,z)}{4} +\delta r^2
s^2.
\]
With all this in hands, we can now turn to our proof. We first start by
changing the contour of integration in (\ref{integrale}):
\begin{align*}
\int_{-\infty}^{+\infty} e^{-2\lambda} e^{i \lambda z}
\frac{\lambda}{\sinh \lambda } e^{-\frac{r^2}{4}\lambda
\text{cotanh}  \lambda   }d \lambda & = \int_{\mathbf{Im} \lambda
=\theta (\sqrt{r^2-8},z)} e^{-2\lambda} e^{i \lambda z}
\frac{\lambda}{\sinh \lambda } e^{-\frac{r^2}{4}\lambda
\text{cotanh}  \lambda   }d \lambda \\
&=
\int_{\mathbf{Im} \lambda =\theta (\sqrt{r^2-8},z)}  e^{i \lambda
z} \frac{\lambda}{\sinh \lambda } e^{-\left( \frac{r^2}{4}-2
\right) \lambda \text{cotanh} \lambda   }e^{ 2\lambda-2  \lambda
\text{cotanh} \lambda   } d \lambda
\end{align*}
Therefore, by denoting
\[
U(\lambda)=e^{
2\lambda-2 \lambda \text{cotanh} \lambda } \frac{\lambda}{\sinh
\lambda }
\]
we get
\begin{align*}
 \left| \int_{-\infty}^{+\infty} e^{-2\lambda} e^{i \lambda z}
\frac{\lambda}{\sinh \lambda } e^{-\frac{r^2}{4}\lambda
\text{cotanh}  \lambda   }d \lambda \right|  \le &
e^{-\frac{d(\sqrt{r^2-8},z)^2}{4}}\int_{\mid s \mid \le 1 }
e^{-(r^2-8) \delta^2s^2}
\left|U(s+i\theta(\sqrt{r^2-8},z)) \right|ds  \\
    &+e^{-\frac{d(\sqrt{r^2-8},z)^2}{4}} \int_{\mid s \mid \ge 1
    }e^{-(r^2-8) \delta^2}
\left|U(s+i\theta(\sqrt{r^2-8},z)) \right|ds \\
\le & C_1 \frac{e^{-\frac{d(r,z)^2}{4}}}{r},
\end{align*}
where we used the fact that on the domain on which we work, the
difference $d(\sqrt{r^2-8},z)-d(r,z)$ is uniformly bounded.
Finally, from the lower estimate of \cite{Li2}, on the considered
domain,
\[
p_t (r,z) \ge C_2 \frac{e^{-\frac{d(r,z)^2}{4}}}{r}.
\]
It concludes the proof of step 1.

\textbf{Step 2.} We show that there exists $\eta >0$ such that
\[
\sup_{ \mid z \mid \ge 1 , r^2 \le \eta \mid z \mid} \frac{\left|
p_1 (r, z+2i) \right|}{p_1 (r,z)} < +\infty.
\]
We first start by giving an analytic representation of
\[
p_1 (r, z+2i) 
\]
that is valid on the domain on which we work. As in the previous proof, we
assume $z >0$. Due to the Cauchy theorem, we can change the contour of
integration in the representation (\ref{gaveau}), to get with $0 < \varepsilon
< \pi$,
\begin{align*}
p_1 (r,z) & =\frac{1}{8 \pi^2} \sum_{k=1}^{+\infty} \int_{\mid
\lambda -ik\pi \mid = \varepsilon} e^{i \lambda z}
\frac{\lambda}{\sinh \lambda } e^{-\frac{r^2}{4}\lambda
\text{cotanh}  \lambda   }d \lambda \\
& =\frac{-i}{8 \pi^2} \sum_{k=1}^{+\infty} \int_{\mid
\lambda \mid = \varepsilon} e^{i (-i\lambda+ik\pi) z}
\frac{(-i\lambda+ik\pi)}{\sinh (-i\lambda+ik\pi)}
e^{-\frac{r^2}{4}(-i\lambda+ik\pi) \text{cotanh} (-i\lambda+ik\pi)
}d \lambda \\
& =\frac{-i}{8 \pi^2}  \int_{\mid
\lambda \mid = \varepsilon}  \frac{ e^{-(\pi-\lambda)\left( z
-\frac{r^2}{4} \text { cotan } \lambda \right)}}{1+e^{-\pi\left( z
-\frac{r^2}{4} \text { cotan } \lambda  \right)} }
\left(
\frac{\pi}{1+e^{-\pi\left( z -\frac{r^2}{4} \text { cotan }
\lambda  \right) } }-\lambda \right) \frac{d\lambda}{\sin \lambda}
\end{align*}
Therefore, for $z>0$,
\begin{align*}
 & p^*_1 (r, z+2i) +\frac{1}{4 \pi^2 \left( -1+iz
+\frac{r^2}{4}\right)^2} \\
= & \frac{-i}{8 \pi^2}  \int_{\mid \lambda \mid = \varepsilon}
e^{2i\lambda} \frac{ e^{-(\pi-\lambda)\left( z -\frac{r^2}{4}
\text { cotan } \lambda \right)}}{1+e^{-\pi\left( z -\frac{r^2}{4}
\text { cotan } \lambda  \right)} }
\left(
\frac{\pi}{1+e^{-\pi\left( z -\frac{r^2}{4} \text { cotan }
\lambda  \right) } }-\lambda \right) \frac{d\lambda}{\sin \lambda}
\end{align*}
On our domain, if $\eta$ is small enough, when $r, z \rightarrow
+\infty$, $\mathbf{Re}(z -\frac{r^2}{4} \text { cotan } \lambda)$ goes
uniformly on the circle $\mid \lambda \mid =\varepsilon$ to
$+\infty$. Consequently, on our domain
\begin{multline*}
\left| p^*_1 (r, z+2i) +\frac{1}{4 \pi^2 \left( -1+iz
+\frac{r^2}{4}\right)^2}\right| %\\
\le c_1 \left| \int_{\mid \lambda
\mid = \varepsilon} e^{2i\lambda} e^{-(\pi-\lambda)\left( z
-\frac{r^2}{4} \text { cotan } \lambda \right)}
\left(\pi - \lambda \right) \frac{d\lambda}{\sin \lambda}
\right|
\end{multline*}
for some finite positive constant $c_1$. By choosing $\varepsilon= \pi- 2
\theta (r,z)$, we have
\begin{align*}
 & \int_{\mid \lambda \mid = \varepsilon} e^{2i\lambda}
e^{-(\pi-\lambda)\left( z -\frac{r^2}{4} \text { cotan } \lambda
\right)}
\left(\pi - \lambda \right) \frac{d\lambda}{\sin \lambda} \\
= & \int_{\mid \lambda
\mid = \pi -2\theta (r,z)} e^{2i\lambda} e^{-(\pi-\lambda)\left( z
-\frac{r^2}{4} \text { cotan } \lambda \right)}
\left(\pi - \lambda \right) \frac{d\lambda}{\sin \lambda},
\end{align*}
where the function $\theta (r,z)$ has been introduced above. At this stage, we
can follow step by step the proof of Theorem 2.17 in \cite{Gaveau2} (the only
difference is in the function $V$ which we take equal to
$V(\lambda)=e^{2i\lambda}\frac{\pi -\lambda}{\sin \lambda})$ to get an
estimate on our domain :
\[
\left| \int_{\mid \lambda \mid = \pi -2\theta (r,z)} e^{2i\lambda}
e^{-(\pi-\lambda)\left( z -\frac{r^2}{4} \text { cotan } \lambda
\right)}
\left(\pi - \lambda \right) \frac{d\lambda}{\sin \lambda} \right|
\le c_2 \frac{e^{-\frac{d(r,z)^2}{4}}}{\sqrt{ r d(r,z)}}
\]
for some finite positive constant $c_2$. Finally, the lower estimate of
\cite{Li2} leads to the conclusion.
\end{proof}

\begin{rmk}
 In order to extend the H.-Q. Li inequality to more general situations, it would be
 interesting to get a proof of the above proposition that would not use the
 explicit expression for $p_t (r,z)$.
\end{rmk}

We can now reprove H.-Q. Li's inequality by using a partition of the unity
(which is here simpler than in the previous subsection) and the $L^1-$Poincar\'e inequality of 
lemma \ref{le:cheloc} (which was also used in the previous subsection). Let $f:\mathbb{H} \rightarrow \mathbb{R}$ be a smooth
positive function compactly supported and let $0 \le \phi \le 1$ be a smooth
function that takes the value $1$ on a ball $\mathbf{B}_{R_1}$ and the value
$0$ outside the ball $\mathbf{B}_{R_2}$ where $R_1 < R_2$, with $R_1$ big
enough. We have
\begin{align*}
(X+iY)P_1 f (0)=&(X+iY)P_1 \phi f (0)+(X+iY)P_1 (1-\phi)f (0) \\
=&\int_{\mathbb{H}} p^*_1 (r,z+2i) (X+iY)(f\phi) (r,\theta,z) rdrd\theta dz \\
& +\int_{\mathbb{H}} p_1 (r,z+2i) (X+iY)(f(1-\phi)) 
(r,\theta,z) rdrd\theta dz \\
=&\int_{\mathbb{H}}\phi (r,\theta,z) p^*_1 (r,z+2i) (X+iY)f (r,\theta,z) rdrd\theta dz \\
& +\int_{\mathbb{H}}(1-\phi(r,\theta,z)) p_1 (r,z+2i) (X+iY)f 
(r,\theta,z) rdrd\theta dz \\
 & +\frac{1}{4\pi^2}\int_{\mathbb{H}} f (r,\theta,z) \frac{(X+iY)\phi (r,\theta,z)}{\left(-1+iz+\frac{r^2}{4} \right)^2}rdrd\theta dz.
\end{align*}
Therefore
\[
\mid \nabla P_1 f (0) \mid \le C  P_1\mid \nabla f  \mid (0)+\left| \frac{1}{4\pi^2}\int_{\mathbb{H}} f (r,\theta,z) \frac{(X+iY)\phi (r,\theta,z)}{\left(-1+iz+\frac{r^2}{4} \right)^2}rdrd\theta dz \right|.
\] 
Now, we estimate $\left| \int_{\mathbb{H}} f(r,\theta,z) \frac{(X+iY)\phi
   (r,\theta,z)}{\left(-1+iz+\frac{r^2}{4} \right)^2}rdrd \theta dz \right|$
thanks to lemma \ref{le:cheloc}:
\begin{align*}
& \left| \int_{\mathbb{H}} f(r,\theta,z)  \frac{(X+iY)\phi (r,\theta,z)}{\left(-1+iz+\frac{r^2}{4} \right)^2}rdrd
\theta dz \right|\\
=&\left| \int_{\mathbb{H}} ( f(r,\theta,z) -m)  \frac{(X+iY)\phi (r,\theta,z)}{\left(-1+iz+\frac{r^2}{4} \right)^2}rdrd
\theta dz \right| & (m \text{ is the mean of } f \text{ on } \mathbf{B}_{R_2}) \\
\le & C_1   \int_{\mathbf{B}_{R_2} } \mid f(r,\theta,z) -m  \mid  rdrd\theta dz  \\
 \le & C_2 \int_{\mathbf{B}_{R_2} } \mid \nabla f \mid (r,\theta,z) rdrd\theta dz\\
\le & C_3 P_1\mid \nabla f  \mid (0).
\end{align*}
This completes the proof of H.-Q. Li's inequality. 

\
As we mentioned it in the
beginning of this section, interestingly, it is not possible to find a
function $\phi$ on $\He$ such that:
\begin{itemize}
\item $(X+iY) \phi =0$;
\item The ratio $\frac{ \left| p_1^* (r,z+2i) -\Phi (r,\theta, z) \right|
 }{p_1(r,z)}$ is bounded.
\end{itemize}
Indeed, the first point implies that $\Phi$ can be written:
\[
\Phi (r,\theta, z)=H \left( \frac{r^2}{4}+iz, r e^{i\theta} \right),
\]
where $H: \{ z_1 \in \mathbb{C}, \mathbf{Re} (z_1) \ge 0 \} \times \mathbb{C}
\rightarrow \mathbb{C}$ is analytic in $z_1$ and $z_2$. Now, due to the
estimate of Proposition \ref{estimate_ratio} and the estimate on $p_1$, it
would imply that for $r$ and $z$, such that $r^2+\mid z\mid$ is big enough:
\[
\left| H \left( \frac{r^2}{4}+iz, r e^{i\theta} \right) +\frac{1}{4 \pi^2 \left( -1+iz
+\frac{r^2}{4} \right)^2} \right| \le A e^{-B(r^2+\mid z \mid)}
\]
where $A$ and $B$ are strictly positive constants. Now, we have the following lemma that prevents 
the existence of such $H$:

\begin{lem}
Let $f: \{ z_1 \in \mathbb{C}, \mathbf{Re} (z_1) \ge 0 \} \times \mathbb{C} \rightarrow \mathbb{C}$ 
be analytic in $z_1$ and $z_2$. If there exist strictly positive constants $A$ and $B$ such that
\[
\forall r \ge 0,~\forall z \in \mathbb{R},~\forall \theta \in [0,2\pi],\quad
\left| f \left( r^2+iz, r e^{i\theta} \right) \right| \le A e^{-B(r^2+\mid z \mid)}
\]
then $f=0$.
\end{lem}

\begin{proof}

Let $ r \ge 0, z \in \mathbb{R}$. The function $z_2 \rightarrow  f \left( r^2+iz, z_2 \right)$ is analytic,
therefore from the maximum principle we have
\[
\left| f \left( r^2+iz, z_2 \right) \right| \le A e^{-B(\mid z_2 \mid^2+\mid z \mid)}, 
\]
for $\mid z_2 \mid \le r$. Consequently, on the set $\mathbf{Re} (z_1) \ge \mid z_2 \mid^2$ we have
\[
\left| f \left( z_1, z_2 \right) \right| \le A e^{-B(\mid z_2 \mid^2+\mid \mathbf{Im} (z_1) \mid)}.
\]
By using the analytic function $z_1 \rightarrow f(z_1,z_2)$, a translation, and a multiplication by
$e^{-z_1}$we would therefore obtain
a function $g$ analytic on the  set $\mathbf{Re} (z) > 0$ such that
\[
\left| g(z) \right| \le \alpha e^{-\beta \mid z \mid }
\]
with $\alpha , \beta >0$, and such function has clearly to be $0$ (Use for
instance the conformal equivalence between the set $\mathbf{Re} (z) > 0$ and
the open unit disc to get a function $h$ analytic on the disc that satisfy the
estimate $\left| g(z) \right| \le \alpha' e^{-\frac{\beta'} {\mid z \mid} }$).
\end{proof}

\section{Functional inequalities for the heat kernel}
\label{se:inegs}

Most of the consequences of the classical gradient bounds under a $\Gad$
curvature assumption remain true under an H.-Q. Li gradient bound. In the
sequel, we derive, by interpolation from the gradient bound \eqref{HQLi},
several local functional inequalities of Gross-Poincar\'e-Cheeger-Bobkov type
for the heat kernel on the Heisenberg group. The term \emph{local} means that
these inequalities concern the probability measure $P_t(\cdot)(\Bx)$ at fixed
$t$ and $\Bx$, in contrast to inequalities for the invariant measure. These
local inequalities can be seen as global inequalities for Gaussian measures on
the Heisenberg group. In the literature, these inequalities and interpolations
where mainly developed in Riemannian settings under a $\Gad$ curvature
assumption. Rigorously, the semigroup interpolations used in the sequel rely
on the existence of an algebra of functions $\mathcal{A}$ from $\He$ to
$\mathbb{R}$ stable by the action of the heat kernel. Thanks to lemma
\ref{le:A-Algebra}, the Schwartz class $\mathcal{S}$ of smooth and rapidly
decreasing functions in $\R^{3}$ may play this role in the case of the
Heisenberg group $\He$.

\subsection{Gross-Poincar\'e type inequalities}

One of the first consequence of the gradient bound \eqref{HQLi} is
Gross-Poincar\'e type local inequalities, also called $\varphi$-Sobolev
inequalities in \cite{kyoto,hu}. Namely, let $\varphi:I\to\mathbb{R}$ be a
smooth convex function defined on an open interval $I\subset\mathbb{R}$ such
that $\varphi''>0$ on $I$ and $-1/\varphi''$ is convex on $I$.
% Assume that the class of functions of the form $\varphi\circ f$ where
% $f:\He\to I$ with $f\in\mathcal{A}$ is left globally stable by $P_t$ for any
% $t\geq0$.

\begin{thm}[Local Gross-Poincar\'e inequalities]
 By using the notations of \eqref{HQLi}, for every $t\geq0$, every
 $\Bx\in\He$, and every $f\in\cC(\He,I)$,
 \begin{equation}\label{eq:phisob}
   % \mathbf{Ent}_{P_t(\cdot)(\Bx)}^\varphi(f)=
   P_t(\varphi(f))-\varphi(P_t f) %
   \leq C_1^2\,t\, P_t \left(\varphi''(f)|\nabla f\right|^2).
 \end{equation}
\end{thm}

\begin{proof}
 One can assume that the support of $f$ is strictly included in $I$. Since
 $L$ is a diffusion operator, $L(\alpha(f))=\alpha'(f)Lf+\alpha''(f)\Ga f$
 for any $f\in\cC(\He,\mathbb{R})$ and any smooth
 $\alpha:\mathbb{R}\to\mathbb{R}$. By the semigroup and the diffusion
 properties,
 \begin{align*}
   P_t(\varphi(f))-\varphi(P_t f)
   =\int_0^t\!\partial_s P_s(\varphi(P_{t-s} f))\,ds 
   =\int_0^t\!P_s(\varphi''(P_{t-s} f)\left|\nabla P_{t-s}f\right|^2)\,ds.
 \end{align*}
 Now, \eqref{HQLi} gives $\left|\nabla P_{t-s} f\right|^2 \leq C_1^2
 (P_{t-s}(|\nabla f|))^2$. Next, by the Cauchy-Schwarz inequality or
 alternatively by the Jensen inequality for the bivariate convex function
 $(u,v)\mapsto \varphi''(u)v^2$, we get $\varphi''(P_{t-s} f)(P_{t-s}(|\nabla
 f|))^2 %
 \leq P_{t-s}(\varphi''(f)|\nabla f|^2)$, which gives the desired result.
\end{proof}

\begin{itemize}
\item for $\varphi(u)=u\log(u)$ on $I=(0,\infty)$, we get a Gross logarithmic
 Sobolev inequality, mentioned for instance in \cite{HQLi} (see also
 \cite{gross,gross-varenna}),
 \begin{equation}\label{eq:lsi}
   P_t(f\log(f))-P_t(f)\log(P_t(f)) %
   \leq C_1^2\,t\,P_t\left(f^{-1}\left|\nabla f\right|^2\right);
 \end{equation}
\item for $\varphi(u)=u^p$ on $I=(0,\infty)$ with $1<p\leq 2$, we get a
 Beckner-Lata{\l}a-Oleszkiewicz type inequality (see
 \cite{beckner,latala-oleszkiewicz})
 \begin{equation}\label{eq:blo}
   \frac{P_t(f^p)-(P_t(f))^p}{p-1} %
   \leq p\,C_1^2\,t\,P_t(f^{p-2}|\nabla f|^2);
 \end{equation}
\item for $\varphi(u)=u^2$ on $I=\mathbb{R}$, we get a Poincar\'e inequality,
 mentioned in \cite{DriverMelcher},
 \begin{equation}\label{eq:pi}
   P_t(f^2)-(P_t(f))^2 %
   \leq 2\,C_1^2\,t\,P_t(\left|\nabla f\right|^2).
 \end{equation}
\end{itemize}

We have seen in the introduction that a local Poincar\'e inequality such as
\eqref{eq:pi} can be also obtained from the Driver and Melcher gradient bound
\eqref{DM}, with a constant $2C_2$ instead of $2C_1^2$. However, the
inequalities \eqref{eq:lsi} and \eqref{eq:blo} need the stronger gradient
bound \eqref{HQLi} of H.-Q. Li. They also imply the local Poincar\'e
inequality \eqref{eq:pi} by linearization. % , see for instance \cite{kyoto}.
It is shown in \cite[Theorem 4.4]{esaim} that the convexity of the bivariate
function $(u,v)\mapsto \varphi''(u)v^2$ is equivalent to the convexity of the
$\varphi$-entropy functional and also to the tensorization property of the
$\varphi$-entropy functional. This fact is related to the infinite dimensional
nature of \eqref{eq:phisob}. The inequality \eqref{eq:blo} interpolates
between \eqref{eq:lsi} (let $p\to1^+$) and \eqref{eq:pi} (take $p=2$). The
linearity with respect to $t$ of the constant in front of the right hand side
of \eqref{eq:phisob} is related to the fact that $(P_t)_{t\geq0}$ is a
convolution semigroup, namely $P_t(\cdot)(\Bx)$ can be obtained from
$P_1(\cdot)(0)$ by $\Bx$-translation and $\sqrt{t}$-dilation in $\He$.

\subsection{Cheeger type isoperimetric inequalities}

As mentioned in the introduction, it is possible to deduce a reverse local
Poincar\'e inequality from the gradient bounds \eqref{DM} of Driver and
Melcher or \eqref{HQLi} of H.-Q. Li. However, the constants are not known
precisely. A better constant is provided by theorem \ref{th:RSG}, which
implies immediately that for every $t\geq0$ and every
$f\in\cC(\He,\mathbb{R})$,
\begin{equation}\label{eq:grad-bound}
 \left\Vert\left|\nabla P_t f\right|\right\Vert_\infty %
 \leq \frac{1}{\sqrt{t}}\left\Vert f\right\Vert_\infty.
\end{equation}

Cheeger derived in \cite{cheeger} a lower bound for the spectral gap of the
Laplacian on a Riemannian manifold. This bound can be related to a sort of
$\mathrm{L}^1$ Poincar\'e inequality, which has an isoperimetric content, see
\cite{coulhon-saloff} and references therein. Here we derive such an
inequality for the heat kernel by only using the gradient bound \eqref{HQLi},
by mixing arguments borrowed from \cite{bakry-ledoux} and \cite{ledoux-buser}.

\begin{thm}[Local Cheeger type inequality]\label{th:cheeger}
 With the notations of \eqref{HQLi}, for every $t\geq0$, every $\Bx\in\He$,
 and every $f\in\cC(\He,\mathbb{R})$,
 \begin{equation}\label{eq:cheeger}
   P_t(\left|f-P_t(f)(\Bx)\right|)(\Bx) %
   \leq 4C_1\sqrt{t}\,P_t(\left|\nabla f\right|)(\Bx).
 \end{equation}  
\end{thm}

\begin{proof}
 We adapt the method used in \cite[p. 953]{ledoux-buser} for the invariant
 measure in Riemannian settings. For any $g\in\cC(\He,\mathbb{R})$
 with $\left\Vert g\right\Vert_\infty\leq 1$, any $t\geq0$, and any
 $\Bx\in\He$,
 \begin{align*}
   P_t((f-P_t(f)(\Bx))g)(\Bx)
   &=P_t(fg)(\Bx)-P_t(f)(\Bx)P_t(g)(\Bx) \\
   &=\int_0^t\!\partial_s P_s((P_{t-s} f)(P_{t-s}g))(\Bx)\,ds \\
   &=2\int_0^t\!P_s(\Ga(P_{t-s} f,P_{t-s} g))(\Bx)\,ds \\
   &\leq 2\int_0^t\! P_s(\left|\nabla P_{t-s} f \right|
   \left|\nabla P_{t-s} g \right|)(\Bx)\,ds \\
   &\leq 2C_1\,P_t(\left|\nabla f \right|)(\Bx)
   \int_0^t\!\frac{\left\Vert g\right\Vert_\infty}{\sqrt{(t-s)}}\,ds \\
   &\leq 4C_1\sqrt{t}\,P_t(\left|\nabla f\right|)(\Bx).
 \end{align*}
 where we used the gradient bound \eqref{HQLi} for $f$ and the gradient bound
 \eqref{eq:grad-bound} for $g$. The desired result follows then by
 $\mathrm{L}^1-\mathrm{L}^\infty$ duality by taking the supremum over $g$.
 % See also \cite[th. 4.1 p. 275 and subsequent comments]{bakry-ledoux}.
\end{proof}

Similarly, we get also the following correlation bound for every $t\geq0$ and
$f,g\in\cC(\He,\mathbb{R})$,
\begin{equation}\label{eq:corr}
 \left|P_t(fg)-P_t(f)P_t(g)\right| %
 \leq 2C_1^2t\,
 \sqrt{P_t(\left|\nabla f\right|^2)}
 \sqrt{P_t(\left|\nabla g\right|^2)}.
\end{equation}
%It is also possible to derive a lower bound of FKG type. 
When $f=g$, we recover the Poincar\'e inequality \eqref{eq:pi}.

\begin{thm}[Yet another local Cheeger type inequality]\label{th: cheeger2}
 With the notations of \eqref{HQLi}, for every $t\geq0$, every $\Bx\in\He$,
 and every ball $B$ of $\He$ for the Carnot-Carath\'eodory metric, there
 exists a real constant $C_{B,t,\Bx}>1$ such that for every function
 $f\in\cC(\He,\mathbb{R})$ which vanishes on $B$,
 \begin{equation}\label{eq:cheeger-bis}
   \left|P_t(f)(\Bx)\right| %
   \leq C_{B,t,\Bx}\,P_t(\left|\nabla f\right|)(\Bx).
 \end{equation}    
\end{thm}

\begin{proof}
 Let $g\in\mathcal{C}^\infty(\He,\mathbb{R})$ be such that $\left\Vert
   g\right\Vert_\infty<\infty$ and $g\equiv 1$ on $B^c$. Since $fg=f$, the
 computation made in the proof of theorem \ref{th:cheeger} provides
 $$
 P_t(f)(\Bx)-P_t(f)(\Bx)P_t(g)(\Bx) %
 \leq 4C_1\sqrt{t}\,\Vert g\Vert_\infty \,P_t(\left|\nabla f\right|)(\Bx).
 $$
 For any arbitrary real number $r\geq1$, the class of functions
 $$
 \mathcal{C}_{B,r}=\{g\in\mathcal{C}^\infty(\He,\mathbb{R}) \text{
   with $\Vert g\Vert_\infty\leq r$ and $g\equiv1$ on $B^c$}\}.
 $$
 is not empty since it contains the constant function $\equiv 1$.
 Furthermore, since $P_t(\cdot)(\Bx)$ is a probability measure with non
 vanishing density, the following extrema
 $$
 \alpha_-(B,r,t,\Bx)=\inf_{g\in\mathcal{C}_{B,r}}P_t(g)(\Bx)
 \quad\text{and}\quad
 \alpha_+(B,r,t,\Bx)=\sup_{g\in\mathcal{C}_{B,r}}P_t(g)(\Bx)
 $$
 are finite and non zero. Moreover, an elementary local perturbative argument
 on any element of the class $\mathcal{C}_{B,r}$ shows that
 $\alpha_-(B,r,t,\Bx)\,\alpha_+(B,r,t,\Bx)<0$ as soon as $r$ is large enough,
 say $r\geq r_{B,t,\Bx}$. Thus, $P_t(f)(\Bx)P_t(g)(\Bx)\leq0$ for some
 $g\in\mathcal{C}_{B,r}$. The desired result follows then with
 $C_{B,t,\Bx}=4C_1\sqrt{t}\,r_{B,t,\Bx}$, since one can replace $f$ by $-f$
 in the obtained inequality. Note that $C_{B,t,\Bx}$ blows up when
 $\mathrm{vol}(B)\searrow0$.
 % Moreover, $B$ does not need to be a ball at all.
 Actually, this proof does not use the nature of the Heisenberg group $\He$,
 and relies roughly only on the diffusion property, the smoothness of the
 heat kernel and the gradient bound. However, on the Heisenberg group $\He$,
 the usage of translations and dilations and of the convolution semigroup
 nature of $(P_t)_{t\geq0}$ allows to precise the dependency of $C_{B,t,\Bx}$
 over $t$ and $\Bx$ by using $\Bx$-translation and $\sqrt{t}$-dilation.
\end{proof}

The isoperimetric content of \eqref{eq:cheeger} can be extracted by
approximating an indicator with a smooth $f$, see for instance
\cite{bakry-ledoux}. Namely, for any Borel set $A\subset\He$ with smooth
boundary, any $t\geq0$, and any $\Bx\in\He$, we get by denoting
$\mu_{t,\Bx}=P_t(\cdot)(\Bx)$,
\begin{equation}\label{eq:isop-set-cheeger-1}
 \mu_{t,\Bx}(A)(1-\mu_{t,\Bx}(A)) %
 \leq 2C_1\sqrt{t}\,\mu_{t,\Bx}^\text{surface}(\partial A)
\end{equation}
where $\mu_{t,\Bx}^\text{surface}(\partial A)$ is the perimeter of $A$ for
$\mu_{t,\Bx}$ as defined in \cite[Section 3]{ambrosio} (see also
\cite{monti}). From \eqref{eq:cheeger-bis}, we get similarly for any ball $B$
in $\He$ and any Borel set $A\subset B^c$ with smooth boundary,
\begin{equation}\label{eq:isop-set-cheeger-2}
 \mu_{t,\Bx}(A) \leq C_{B,t,\Bx}\,\mu_{t,\Bx}^\text{surface}(\partial A).
\end{equation}

\subsection{Bobkov type isoperimetric inequalities}

Let $F_\gamma:\mathbb{R}\to[0,1]$ be the cumulative probability function of
the standard Gaussian distribution $\gamma$ on the real line $\mathbb{R}$,
given for every $t\in\mathbb{R}$ by
$$
F_\gamma(t)=\frac{1}{\sqrt{2\pi}}\int_{-\infty}^t\!e^{-\frac{1}{2}u^2}\,du.
$$
The Gaussian isoperimetric function $\mathcal{I}:[0,1]\to[0,(2\pi)^{-1/2}]$ is
defined by $\mathcal{I}=(F_\gamma)'\circ (F_\gamma)^{-1}$. The function
$\mathcal{I}$ is concave, continuous on $[0,1]$, smooth on $(0,1)$, symmetric
with respect to the vertical axis of equation $u=1/2$, and satisfies to the
differential equation
\begin{equation}\label{eq:equadiff-I}
 \mathcal{I}(u)\mathcal{I}''(u)=-1\quad\text{for any $u\in[0,1]$}
\end{equation}
with $\mathcal{I}(0)=\mathcal{I}(1)=0$ and
$\mathcal{I}'(0)=-\mathcal{I}'(1)=\infty$. Note that $\mathcal{I}(u)\geq
u(1-u)$ for any real $u\in[0,1]$, and that $\mathcal{I}(u)\leq \min(u,1-u))$
when $u$ belongs to a neighborhood of $1/2$.

%\begin{figure}[htbp]
%  \begin{center}
%    \includegraphics[scale=0.7]{isopfuncs}
%  \end{center}
%  \caption{Isoperimetric functions comparisons}
%  \label{fi:isopfuncs}
%\end{figure}

\begin{lem}[Yet another uniform gradient bound]
 With the notations of \eqref{HQLi}, for every $t\geq0$ and
 $f\in\cC(\He,(0,1))$,
 \begin{equation}\label{eq:isop}
   \mathcal{I}(P_t f) - P_t(\mathcal{I}(f)) %
   \leq C_1^2\sqrt{2t}\,P_t(\left|\nabla f\right|).
 \end{equation}
\end{lem}

\begin{proof}
 The inequality \eqref{eq:isop} was obtained by Bobkov in \cite{bobkov-old}
 for the standard Gaussian measure on $\R$. Later, it was generalized in
 \cite{bakry-ledoux}, by using semigroup techniques, to Riemannian settings
 under a $\Gad$ curvature assumption. We give here a proof by adapting the
 argument given in \cite[p. 261-263]{bakry-ledoux} from invariant measure
 settings to local settings. One may assume that $\varepsilon\leq f\leq
 1-\varepsilon$ for some $\varepsilon>0$. By the diffusion property and
 \eqref{eq:equadiff-I}
 \begin{align*}
   \left[\mathcal{I}(P_t f)\right]^2-\left[P_t(\mathcal{I}(f))\right]^2
   &=-\int_0^t\!\partial_s\!\left[P_s(\mathcal{I}(P_{t-s} f))\right]^2\,ds \\
   &=-2\int_0^t\!P_s(\mathcal{I}(P_{t-s} f))
   P_s\left(\mathcal{I''}(P_{t-s} f)\left|\nabla P_{t-s} f\right|^2\right)\,ds\\
   &=+2\int_0^t\!P_s(\mathcal{I}(P_{t-s} f)) P_s\left(\frac{\left|\nabla
	  P_{t-s} f\right|^2}{\mathcal{I}(P_{t-s} f)}\right)\,ds.
 \end{align*}
 Next, the Cauchy-Schwarz inequality or alternatively the Jensen inequality
 for the bivariate convex function $(u,v)\mapsto
 u^2/\mathcal{I}(v)=-\mathcal{I''}(v)u^2$ gives
 $$
 \left[\mathcal{I}(P_t f)\right]^2-\left[P_t(\mathcal{I}(f))\right]^2
 \geq 2\int_0^t\!\left[P_s(|\nabla P_{t-s} f|)\right]^2\,ds.
 $$
 Now by using the gradient bound \eqref{HQLi} we have
 $$
 C_1\,P_s(|\nabla
 P_{t-s} f|)\geq |\nabla P_s(P_{t-s} f)|=|\nabla P_t f|
 $$
 and thus
 $$
 \left[\mathcal{I}(P_t f)\right]^2-\left[P_t(\mathcal{I}(f))\right]^2
 \geq \frac{2t}{C_1^2}\,|\nabla P_t f|^2.
 $$
 In particular, we obtain the following uniform gradient bound
 $$
 \left\Vert \mathcal{I}''(P_tf)|\nabla P_t f|\right\Vert_\infty
 =\left\Vert\frac{|\nabla P_t f|}{\mathcal{I}(P_tf)}\right\Vert_\infty
 \leq \frac{C_1}{\sqrt{2t}}.
 $$
 We are now able to prove \eqref{eq:isop}. By the diffusion property
 $$
 \mathcal{I}(P_t f)-P_t(\mathcal{I}(f))
 = -\int_0^t\!\partial_s\,P_s(\mathcal{I}(P_{t-s} f))\,ds 
 = -\int_0^t\!
 P_s(\mathcal{I}''(P_{t-s} f)\left|\nabla P_{t-s}f\right|^2)\,ds.
 $$ 
 By \eqref{HQLi} we get $|\nabla P_{t-s}f|^2\leq C_1\, |\nabla P_{t-s}
 f|P_{t-s}(|\nabla f|)$ and thus
 $$
 \mathcal{I}(P_t f)-P_t(\mathcal{I}(f))
 \leq %
 C_1\,%
 \left(\int_0^t\!\frac{C_1}{\sqrt{2(t-s)}}\,ds\right)P_t(|\nabla f|)
 =
 C_1^2\sqrt{2t}\, P_t(|\nabla f|).
 $$
 %\begin{flushright}\qed\end{flushright}
\end{proof}

The isoperimetric content of \eqref{eq:isop} can be extracted by approximating
an indicator with a smooth $f$, see \cite{bakry-ledoux}. Namely, for any Borel
set $A\subset\He$ with smooth boundary, any $t\geq0$, and any $\Bx\in\He$, we
get by denoting $\mu_{t,\Bx}=P_t(\cdot)(\Bx)$,
\begin{equation}\label{eq:isop-set-bobkov}
 \mathcal{I}(\mu_{t,\Bx}(A))
 \leq C_1^2\sqrt{2t}\,\mu_{t,\Bx}^\text{surface}(\partial A).
\end{equation}

\begin{cor}[Yet another local Bobkov Gaussian isoperimetric inequality]
 With the notations of \eqref{HQLi}, for every $t\geq0$ and
 $f\in\cC(\He,(0,1))$,
 \begin{equation}\label{eq:isop-tenso-heis}
   \mathcal{I}(P_t f)
   \leq P_t\left(\sqrt{(\mathcal{I}(f))^2+2C^4t\,|\nabla f|^2}\right).
 \end{equation}  
\end{cor}

\begin{proof}
  The desired result follows from the transportation-rearrangement argument
  given in \cite[prop. 5 p. 427]{barthe-maurey}, which is inspired from
  \cite[p. 273]{bakry-ledoux}. The method is not specific to the heat
  semigroup on the Heisenberg group. It is based in particular on a similar
  inequality for the standard Gaussian measure on $\R$ obtained by Bobkov in
  \cite{bobkov}.
\end{proof}

One of the most important aspect of \eqref{eq:isop-tenso-heis} is its
stability by tensor product, in contrast with \eqref{eq:isop}, while
maintaining the same isoperimetric content. Moreover, one may recover from
\eqref{eq:isop-tenso-heis} the Gross logarithmic Sobolev inequality
\eqref{eq:lsi} by using the fact that $\mathcal{I}'(u)\sim\sqrt{-2\log(u)}$
and $\mathcal{I}(u)\sim u\sqrt{-2\log(u)}$ at $u=0$. We ignore if
\eqref{eq:isop-tenso-heis} can be obtained directly by semigroup
interpolation, as for the elliptic case in \cite{bakry-ledoux}. The proof
given in \cite{bakry-ledoux} for the elliptic case is based directly on a
curvature bound at the level of the infinitesimal generator, which is not
implied by the gradient bound \eqref{HQLi} on $\He$. We ignore also if one
can adapt on the Heisenberg group the two points space approach used in
\cite{bobkov} or the martingale representation approach used in
\cite{barthe-maurey,capitaine,hsu,ledoux-zurich}. There is a lack of a direct
proof of \eqref{eq:isop-tenso-heis} on the Heisenberg group, despite the fact
that \eqref{eq:isop-tenso-heis} and \eqref{eq:isop} are equivalent, according
to the argument of Barthe and Maurey in \cite[prop. 5 p. 427]{barthe-maurey}.

\begin{rmk}[Abstract Markov settings]
 In fact, up to specific constants, most of the proofs given above have
 nothing to do with the group structure of the space or with the convolution
 semigroup nature of $(P_t)_{t\geq0}$. They remain actually valid in very
 general settings provided that the computations make sense. The key points
 are a $\sqrt{\Ga}-P_t$ sub-commutation and the semigroup and diffusion
 properties. Formally, let $L$ be a diffusion operator on a smooth complete
 connected differential manifold $\mathcal{M}$, generating a Markov semigroup
 $(P_t)_{t\geq0}=(e^{t L})_{t\geq0}$ with smooth density with respect to some
 reference Borel measure on $\mathcal{M}$. Let $2\Ga f=L(f^2)-2fLf$ and
 suppose that there exists $C:(0,\infty)\to(0,\infty)$ such that
 \begin{equation}\label{eq:gamcom}
   \sqrt{\Ga P_t f}\leq C(t)\,P_t(\sqrt{\Ga f})
 \end{equation}
 pointwise for every $t\geq0$ and every smooth $f:\mathcal{M}\to\mathbb{R}$.
 Let us define $R(t)$ by
 $$
 R(t)= \int_0^t\!
 C(s)\left(\int_0^{s}\!\frac{2}{C(u)^2}\,du\right)^{-\frac{1}{2}}\,ds.
 $$
 Then for every $t\geq0$, every $x\in\mathcal{M}$, and every smooth
 $f:\mathcal{M}\to\mathbb{R}$,
 \begin{equation}\label{eq:cheeger-generic}
   P_t(\left|f-P_t(f)(x)\right|)(x)
   \leq 2R(t)\,P_t(\sqrt{\Ga f})(x).
 \end{equation}
 Moreover, for every $t\geq0$ and every smooth $f:\mathcal{M}\to(0,1)$,
 \begin{equation}\label{eq:bobkov-generic}
   \mathcal{I}(P_t(f))-P_t(\mathcal{I}(f)) %
   \leq R(t)\,P_t(\sqrt{\Ga f}),
 \end{equation}
 and
 \begin{equation}\label{eq:bobkov-tenso-generic}
   \mathcal{I}(P_t(f)) %
   \leq P_t\left(\sqrt{(\mathcal{I}(f))^2+R(t)^2\,\Ga f}\right),
 \end{equation}
 where $\mathcal{I}$ stands for the Gaussian isoperimetric function as in
 \eqref{eq:equadiff-I}. Furthermore, if $I$ is an open interval of
 $\mathbb{R}$ and $\varphi:I\to\mathbb{R}$ is a smooth convex function such
 that $\varphi''>0$ on $I$ and $-1/\varphi''$ is convex on $I$, then for
 every $t\geq0$, every $x\in\mathcal{M}$, and every smooth $f:\mathcal{M}\to
 I$,
 \begin{equation}\label{eq:phisob-generic}
   P_t(\varphi(f))-\varphi(P_t f) %
   \leq \left(\int_0^t\!C(u)^2\,du\right)\,%
   P_t (\varphi''(f)\Ga f).
 \end{equation}
 Finally, if $P_t(\cdot)(x)\to\mu$ weakly as $t\to\infty$ for some $x\in\M$
 and some probability measure $\mu$ on $\M$ then the four inequalities
 (\ref{eq:cheeger-generic}-\ref{eq:phisob-generic}) above hold for $\mu$
 instead of $P_t(\cdot)(x)$. Here the constant in \eqref{eq:cheeger-generic}
 is obtained partly by using a reverse local Poincar\'e inequality deduced
 from \eqref{eq:gamcom}. On the Heisenberg group, we used an alternative
 constant for the reverse local Poincar\'e inequality, which was not deduced
 from \eqref{eq:gamcom}.
\end{rmk}

\subsection{Multi-times inequalities}

Let $\varphi:\mathcal{I}\to\mathbb{R}$ be fixed and as in \eqref{eq:phisob}.
The $\varphi$-entropy functional
$$
\mathbf{Ent}_\mu:f\mapsto\mathbf{Ent}_\mu(f) %
=\int\!\varphi(f)\,d\mu -\varphi\left(\int\!f\,d\mu\right)
$$
has the tensor product property. Namely, if
$\mu=\mu_1\otimes\ldots\otimes\mu_n$ is a probability measure on a product
space $E=E_1\times\cdots\times E_n$ then for every $f:E\to\mathcal{I}$ in the
domain of $\mathbf{Ent}_\mu$,
$$
\mathbf{Ent}_\mu(f) 
\leq \sum_{i=1}^n \int\!\mathbf{Ent}_{\mu_i}(f))\,d\mu
$$
where the integrals in $\mathbf{Ent}_{\mu_i}(f)$ act only on the $i^\text{th}$
coordinate. The details are given in \cite{kyoto}. Below, we use the notation
$\mathbf{Ent}(U)=\mathbb{E}(\varphi(U))-\varphi(\mathbb{E}(U))$ for any real
random variable $U$ taking its values in $\mathcal{I}$. Now, let
$(X_t)_{t\geq0}$ be the diffusion process on $\He$ generated by $L$, with
$X_0=0$. Let also $F:\He^n\to\mathcal{I}$ be some fixed smooth function. Here
$\He^n$ stands for the $n$-product space $\He\times\cdots\times\He$. Since
$(X_t)_{t\geq0}$ has independent stationary increments, i.e. is a L\'evy
process on $\He$ associated to a convolution semigroup, we have, for any
finite increasing sequence $0< t_1<\cdots<t_n$ of fixed times,
$$
\mathbf{Ent}(F(X_{t_1},\ldots,X_{t_n}))
=\mathbf{Ent}_{\mathcal{L}(Q_1,\ldots,Q_n)}(F\circ\pi)
$$
where $\pi:\He^n\to\He^n$ is defined by
$$
\pi(\Bx_1,\Bx_2,\ldots,\Bx_n) %
=(\Bx_1,\Bx_1\bullet\Bx_2,\ldots,\Bx_1\bullet\cdots\bullet\Bx_n)
$$
for every $(\Bx_1,\ldots,\Bx_n)\in\He^n$, and where $Q_1,\ldots,Q_n$ are
independent random variables on the Heisenberg group $\He$ with
$\mathcal{L}(Q_i)=\mathcal{L}((X_{t_{i-1}})^{-1}X_{t_i})=\mathcal{L}(X_{t_{i}-t_{i-1}})$
for every $i\in\{1,\ldots,n\}$, with $t_0=0$. The tensor product property of
the entropy given above together with \eqref{eq:phisob} gives
$$
\mathbf{Ent}(F(X_{t_1},\ldots,X_{t_n})) %
\leq
C^2_1\,\mathbb{E}_{\mathcal{L}(X_{t_1},\ldots,X_{t_n})}(\varphi''(F)\mathcal{D}^2_{t_1,\ldots,t_n}
F)
$$
where $C$ is as in \eqref{HQLi} and \eqref{eq:phisob}, and where
$$
\mathcal{D}^2_{t_1,\ldots,t_n} F =\sum_{i=1}^n (t_i-t_{i-1})\,|\nabla_i
(F\circ\pi)|^2\circ\pi^{-1}
% \sum_{i=1}^n (t_i-t_{i-1})\left(\sum_{j=i}^n \nabla_j F\right)^2
$$
where $\nabla_i$ denote the left invariant gradient $\nabla$ of $\He$ acting
on the $i^{\text{th}}$ coordinate $\Bx_i$. Only the distribution of
$\varphi''(F)\mathcal{D}_{t_1,\ldots,t_n}^2 F$ under
$\mathcal{L}(X_{t_1},\ldots,X_{t_n})$ is of interest. Similarly, by using an
argument of Bobkov detailed for instance in \cite[p. 429-430]{barthe-maurey},
we get from \eqref{eq:isop-tenso-heis}, for any smooth function
$F:\He^n\to(0,1)$, by denoting $\nu=\mathcal{L}(X_{t_1},\ldots,X_{t_n})$,
$$
\mathcal{I}(\mathbb{E}_\nu(F)) \leq
\mathbb{E}_\nu\left(\sqrt{(\mathcal{I}(F))^2+2C_2^4\,\mathcal{D}^2_{t_1,\ldots,t_n}
   F}\right).
$$
We ignore if such a cylindrical approach leads to functional inequalities for
the paths space on $\He$, i.e. for the hypoelliptic Wiener measure, by letting
$n\to\infty$. It sounds interesting to try to make a link with
\cite{friz-oberhauser}.

{\footnotesize %

} %footnotesize

\vfil

\noindent
Dominique \textsc{Bakry} \texttt{bakry[@]math.univ-toulouse.fr}\\
Fabrice \textsc{Baudoin} \texttt{fbaudoin[@]math.univ-toulouse.fr}\\
Michel \textsc{Bonnefont} \texttt{bonnefon[@]math.univ-toulouse.fr}\\
Djalil \textsc{Chafa\"\i} \texttt{chafai[@]math.univ-toulouse.fr}

\bigskip

\noindent
\textsc{Institut de Math\'ematiques de Toulouse (CNRS 5219)\\
Universit\'e Paul Sabatier \\
118 route de Narbonne, F-31062 Toulouse, France.}

\end{document}